\documentclass[10pt,journal,twocolumn,oneside,letterpaper,final]{IEEEtran}

\usepackage{amsmath,amssymb,amsthm}
\usepackage{paralist}	
\usepackage{graphicx}
\usepackage{bm} % bold Greek symbols
\usepackage[]{algorithm2e} % algorithms
\usepackage[caption=true]{subfig} % subfigures

\usepackage[colorlinks=true]{hyperref}
\hypersetup{urlcolor=blue, citecolor=red}

\newtheorem{theorem}{Theorem}
\newtheorem{corollary}[theorem]{Corollary}
\newtheorem{remark}[theorem]{Remark}
\newtheorem{definition}[theorem]{Definition}
\newtheorem{lemma}[theorem]{Lemma}

\DeclareMathOperator{\arccosh}{arccosh}

\begin{document}
%\title[ORKA: Object reconstruction using K-approximation]{ORKA: Object reconstruction using a K-approximation graph} 
%\author{Florian Bossmann$^{1,3}$ and Jianwei Ma$^{1,2}$\thanks{$^1$Harbin Institute of Technology, School of Mathematics, Harbin China.\\$^2$Peking University, School of Earth and Space Science, Beijing, China.\\
%$^3$Supported by NSFC grant 42004109 (Seismic inpainting using wave models on scattered data).}}
\title{ORKA: Object reconstruction using a K-approximation graph}
\author{Florian Bossmann$^{1,3}$ and Jianwei Ma$^{1,2}$}
	
\maketitle
\begin{abstract} % 172 words (Inv.Prob: less than 300)
Data processing has to deal with many practical difficulties. Data is often corrupted by artifacts or noise and acquiring data can be expensive and difficult. Thus, the given data is often incomplete and inaccurate. To overcome these problems, it is often assumed that the data is sparse or low-dimensional in some domain. When multiple measurements are taken, this sparsity often appears in a structured manner.

We propose a new model that assumes the data only contains a few relevant objects, i.e., it is sparse in some object domain. We model an object as a structure that can only change slightly in form and continuously in position over different measurements. This can be modeled by a matrix with highly correlated columns and a column shift operator that we introduce in this work.

We present an efficient algorithm to solve the object reconstruction problem based on a $K$-approximation graph. We prove optimal approximation bounds and perform a numerical evaluation of the method. Examples from applications including Geophysics, video processing, and others will be given.
\end{abstract}
\thanks{\noindent$^1$Harbin Institute of Technology, School of Mathematics, Harbin China.\\
$^2$Peking University, School of Earth and Space Science, Beijing, China.\\
$^3$Supported by NSFC grant 42004109 (Seismic data interpolation using wave model decomposition).}

\section{Introduction}

Real world data processing involves many challenges. One of the most common problem that appears in nearly every application is dealing with noisy or incomplete data. Thus, in a first step the data is often processed to remove noise and fill gaps such that the result fits more into the assumed underlying model. This idea can be expressed mathematically as
\begin{align}\label{min1}
	\min\limits_{\bm{X}}\|\bm{D}-\bm{X}\|+\mu\mathcal{R}(\bm{X})
\end{align}
where $\bm{D}$ is the observed real-world data and $\bm{X}$ is the denoised version. The first term, called data fidelity term, controls how much $\bm{X}$ fits to the observed data. Typically, the $2$-norm squared (for vectors) or the Frobenius norm squared (for matrices) is used here. In the second term is called the regularizing term. Here, the operator $\mathcal{R}$ measures the "distance" of $\bm{X}$ to the underlying model and normally only returns positive values (e.g., $\mathcal{R}$ could be a norm of any kind). The operator and model are chosen depending on the application. The parameter $\mu$ balances both terms and should be chosen larger the more noisy the data is. As an alternative to (\ref{min1}) the regularizer can be implemented as inequality
\begin{align}\label{min2}
	\min\limits_{\bm{X}}\|\bm{D}-\bm{X}\|\ \ \ \text{s.t.}\ \ \ \mathcal{R}(\bm{X})\leq\varepsilon
\end{align}
with $\varepsilon\geq0$. While (\ref{min1}) returns the best compromise between model and observation, (\ref{min2}) aims for the best fit to the observation that is within a certain radius to the model space or in the extreme case $\varepsilon=0$ completely fits the model.

One model assumption that has been proven useful in many applications is the concept of sparsity. Here, one assumes that the result $\bm{X}$ only contains a few non-zero elements. Or, more general, one assumes that $\bm{X}$ has a sparse representation in some domain. The latter case can be rewritten as the first case by setting $\bm{X}=\bm{AY}$ with some (linear) operator $\bm{A}$ and then forcing sparsity on $\bm{Y}$. As a typical example, in many applications sparsity under a Wavelet transformation is assumed \cite{Mallat99}. But also low-rank matrices can be considered sparse with respect to their singular values \cite{Zhou14}. Compressed sensing considers sparse representations in more general frames. A good overview can be found in \cite{Foucart13}.

In many applications, the data is not only recorded by one, but by multiple sensors at different positions \cite{Chen06}. Or similarly, the same measurement is not only done once, but repeated at different points in time \cite{Do09}. This way, one gains information about the underlying change over time/position. The data is normally written in a matrix $\bm{D}$ where each column represents one measurement. The above mentioned concepts can easily be adapted by using two dimensional frames such as Curvelets \cite{Ma10} or Shearlets \cite{Kutyniok12}. However, in this multiple measurement scenario the new dimension can be used to expand the sparsity concept. For most applications it is reasonable to assume that the measurements do not change randomly over time/position, but rather follow some underlying laws (e.g., physical restrictions). Thus, the sparsity pattern itself also should not change randomly but is indeed structured. Exemplary, the matrix $\bm{X}$ can be row-sparse \cite{Wu14} (i.e., only few non-zero rows) or block-sparse \cite{Eldar10} (i.e., non-zero entries group together in blocks). This structured sparsity is used in various applications such as machine learning \cite{Huang11,Wen16}, face recognition \cite{Jia12}, and seismic data processing \cite{Wu20}.

We can further classify each model into adaptive and non-adaptive methods. Non-adaptive methods work with a fixed model space while adaptive methods adjust the underlying model to fit more to the data. Exemplary, the model could assume $\bm{X}=\bm{AY}$ and $\bm{Y}$ sparse. Then the operator $\bm{A}$ can be chosen fixed (Wavelets, Curvelets, Shearlets,$\ldots$) or it can be chosen adapted to the data \cite{Yu16,Tosic11}. The latter case is known as dictionary learning. In the multiple measurement case, the underlying sparsity structure can also assumed to be fixed (like row- or block-sparse) or adapted to the data \cite{Plonka10,Bossmann21}. While adaptive methods often show a better performance for the designed application, the adaptivity increases runtime and/or storage costs. Non-adaptive methods on the other hand, are easier to analyze theoretically and often perform good on a wide range of applications.

In a previous work \cite{Bossmann20} the authors introduced a shifted rank-1 model for the multiple measurement case.
%In \cite{Bossmann20} a shifted rank-1 model for the multiple measurement case is introduced.
This model basically combines the idea of low-rank approximation with adaptive structured sparsity and is related to shift-invariant dictionary learning \cite{Rusu13}. The idea in short is as follows.
We assume that each measurement contains the same few objects, but the object position might shift throughout the data. This is modeled as $\mathcal{S}(\bm{uv^*})$ with a rank-1 matrix $\bm{uv^*}$ and a shift operator $\mathcal{S}$ that we will re-introduce in the next section. This model suits various applications. In seismic \cite{Sundman13} and ultrasonic \cite{Basarab13} measurements the same waves (objects) are recorded at different arrival times (positions) depending on the probes position. Video recordings can be interpreted as multiple measurements where each frame is one measurement \cite{Do09}. Then objects (cars, pedestrians,$\ldots$) within the video fit the model assumptions nearly perfectly. Recently, the shifted rank-1 model also has been applied to machine health monitoring \cite{Zhou22}. Other applications are e.g., real-time dynamic MRI \cite{Vaswani10}, dynamic PET \cite{Heins15}, wireless communication \cite{Lian19}, and more \cite{Bossmann20}. A main disadvantage of the shifted rank-1 model is, that the object stays constant over all measurements, i.e., it cannot change form or size. While this is in some cases a reasonable simplification, it does not hold in practice. Cars in videos get smaller as they drive off, seismic waves may vary depending on the subsurface nature, etc..

In this work, we introduce a new model that overcomes these problems. Instead of combining the shift operator with a rank-1 matrix $\bm{uv^*}$, we use a matrix $\bm{U}$ with highly correlated columns. This allows the object to change form and size over the measurements, where the amount of change can be tuned with a regularization parameter. The resulting minimization problem (which is a mix of (\ref{min1}) and (\ref{min2})) can be approximated by constructing a special graph that we call $K$-approximation graph, where $K$ is a constant that defines the approximation rate. We prove that the approximation error decreases exponentially in $K$ but the graph grows exponentially in $K$. The resulting algorithm ORKA (object reconstruction using $K$-approximation) scales only linear in data size, but exponential in $K$. Thus, $K$ must be chosen to balance approximation error and complexity.

The remainder of this work is organized as follows. In the next section we introduce the notations used throughout the work. We also give a short recap on our work \cite{Bossmann20}, which was the motivation of the here presented results. In Section 3 we discuss our model and the according minimization problem. We also reformulate the minimization into a more useful form that can be handled by the algorithm later on. Section 4 derives the $K$-approximation graph for the problem. We also prove approximation error bounds in this section and introduce the ORKA algorithm in the end. Finally, in Section 5 we evaluate the algorithm numerically and demonstrate its usage on different applications.

\section{Notations and Preliminary}

In this Section introduce the notation used and define operators and functions that are needed in the following sections. We also briefly comment on the shifted rank-1 model that we developed in our previous work
\cite{Bossmann20}.

Matrices will be denoted by bold capital letters $\bm{A}$, $\bm{B}$. Vectors are denoted by bold lowercase letters $\bm{u}$, $\bm{v}$. Non-bold lowercase letters represent scalar values, capital letters stand for constants. Finally, operators and functions will be written in calligraphic $\mathcal{S}$.

We use lower indexing to address elements of a matrix or vector. The index is replaced with $:$ whenever the complete column/row is addressed. For example, $\bm{A}_{jk}$  is the element in the $j$-th row and $k$-th column of $\bm{A}$; $\bm{A}_{:k}$ is the $k$-th column of the matrix and $\bm{A}_{j:}$ the $j$-th row. We denote the Frobenius norm for matrices by $\|\cdot\|_F$ and the euclidean norm for vectors by $\|\cdot\|_2$. The according inner product is written as $\langle\cdot,\cdot\rangle$ for both matrices (Frobenius inner product) and vectors (euclidean inner product). For matrices $\bm{A},\bm{B}\in\mathbb{R}^{M\times N}$ and vectors $\bm{u},\bm{v}\in\mathbb{R}^N$ the norms and inner products are defined as
\begin{align*}
	\|\bm{A}\|_F=\sqrt{\sum\limits_{j,k=1}^{M,N}\bm{A}_{jk}^2}, &&
	\langle\bm{A},\bm{B}\rangle=\sum\limits_{j,k=1}^{M,N}\bm{A}_{jk}\bm{B}_{jk},
	\\
	\|\bm{u}\|_2=\sqrt{\sum\limits_{j=1}^{N}\bm{u}_j^2},
	&&
	\langle\bm{u},\bm{v}\rangle=\sum\limits_{j=1}^{N}\bm{u}_j\bm{v}_j.
\end{align*}

Throughout this work, we will denote the given data matrix as $\bm{D}\in\mathbb{R}^{M\times N}$ and the approximation matrix as $\bm{U}\in\mathbb{R}^{M\times N}$. The identity matrix is written as $\bm{I}$. Furthermore, we need the shift matrix
\begin{align*}
	\bm{J}=\begin{pmatrix}
		0 & \cdots & 0 & 1 \\
		1 & \ddots & \vdots & 0 \\
		0 & \ddots & \ddots & \vdots \\
		\ddots &	0 & 1 & 0
	\end{pmatrix}
\end{align*}
and the matrix
\begin{align*}
	\bm{T}=\begin{pmatrix}
		1 & -1 & 0 & \cdots \\
		-1 & 2 & -1 & \ddots \\
		0 & -1 & \ddots & \ddots \\
		\vdots & \ddots & \ddots & \ddots & \ddots \\
		& & & -1 & 2 & -1 \\
		& & &  0 & -1 & 1
	\end{pmatrix}.
\end{align*}
Although $\bm{T}$ appears in a slightly different context in this work, it is actually a second order centered differences matrix with constant boundary conditions.

Last, we introduce the shift operator that we already used in \cite{Bossmann20} (Definition 2.1 and Theorem 2.2).
\begin{definition}\label{def:shiftOperator}
For an integer vector $\bm{\lambda}\in\mathbb{N}^N$, define $\mathcal{S}_{\bm{\lambda}}:\mathbb{R}^{M\times N}\rightarrow\mathbb{R}^{M\times N}$ as
\begin{align*}
\mathcal{S}_{\bm{\lambda}}(\bm{A})=\begin{pmatrix}
\bm{J}^{\bm{\lambda}_1}\bm{A}_{:1} &
\bm{J}^{\bm{\lambda}_2}\bm{A}_{:2} & \cdots &
\bm{J}^{\bm{\lambda}_N}\bm{A}_{:N}
\end{pmatrix}.
\end{align*}
The operator $\mathcal{S}_{\bm{\lambda}}$ shifts the $k$-th column of a matrix by $\bm{\lambda}_k$. The shift operator has the following properties that we need throughout this work:
\begin{align*}
\mathcal{S}_{\bm{\lambda}}^{-1}=\mathcal{S}_{-\bm{\lambda}},
&&
\mathcal{S}_{\bm\lambda}(\bm{A}+\bm{B})=\mathcal{S}_{\bm\lambda}(\bm{A})+\mathcal{S}_{\bm\lambda}(\bm{B}),
\\
\|\mathcal{S}_{\bm{\lambda}}(\bm{A})\|_F=\|\bm{A}\|_F,
&&
\langle\mathcal{S}_{\bm\lambda}(\bm{A}),\mathcal{S}_{\bm\lambda}(\bm{B})\rangle=\langle\bm{A},\bm{B}\rangle,
\\
\mathcal{S}_{\bm\lambda}(\mathcal{S}_{\bm\mu}(\bm{A}))=\mathcal{S}_{\bm\lambda+\bm\mu}(\bm{A}).
\end{align*}
All properties are simple derivations from the definition.
\end{definition}

In our previous work \cite{Bossmann20} we showed,
%In \cite{Bossmann20} it is shown,
that in many applications objects hidden in the data can be approximated by a shifted rank-1 matrix $\mathcal{S}_{\bm{\lambda}}(\bm{uv^*})$. Let us give two examples to clarify this idea. First, consider seismic data that is recorded from several stations at different locations An incoming seismic wave can be described by its oscillation in time $\bm{u}$. This oscillation is measured by all stations with different amplitudes $\bm{v}$ and at different times $\bm{\lambda}$ depending on the distance to the seismic source. Hence, the wave forms approximately a shifted rank-1 matrix. As a second example, consider a video recording where each frame contains a moving object. The object itself can be described by the vector $\bm{u}$, the movement is encoded in the shift vector $\bm{\lambda}$, and the vector $\bm{v}$ gives the visibility or illumination of the object per frame.

The shifted rank-1 model has proven to be a good approximation in several applications. However, it comes with some drawbacks. The vectors $\bm{v}$ and $\bm{\lambda}$ were not regularized and could contain arbitrary jumps or switches in sign. Depending on the application we would actually assume some smoothness on these vectors. Not including this into the shifted rank-1 model can lead to unexpected results and cause problems in presence of high noise. As another problem, the shifted rank-1 model assumes that the object $\bm{u}$ stays constant over all measured data. However, it is more likely that some changes will occur over all measurements. For example, seismic waves change depending on the subsurface material, objects in videos can change their angle or distance to the camera. We want to address these problems with our new approach presented in the following sections.

\section{Model and minimization problem}

In this section, we introduce the new model and derive its minimization problem form that we solve in later sections. Therefore, let $\bm{D}\in\mathbb{R}^{M\times N}$ be the data obtained from $N$ different observations. Throughout this work we assume that the data is real valued. Furthermore, each column of $\bm{D}$ should be obtained from an equidistantly sampled measurement. This ensures that the shift operator $\mathcal{S}_{\bm\lambda}$ does not distort the data because of irregularities in the sampling grid. Last, the distance between two neighboring observations should be constant. Meaning, if the observations were made from position $x_k$, $k=1,\ldots N$, then $|x_k-x_{k+1}|=\text{CONST}$ for $k=1,\ldots,N-1$. Here $x_k$ can be a position in space (e.g., the positions on the surface where seismic or ultrasonic probes were placed) or a position in time (e.g., the time at which a frame of a video was recorded). In our model we assume that any observed object can change form and position from one measurement to the next. The amount of change allowed should depend on the distance between both observations. If this distance stays constant, we can treat each change equally. While the presented methods will also work for more general data $\bm{D}$, these three assumptions massively simplify the notation and analysis.

We first introduce our model idea heuristically. The main concept of our model is, that the data is "object sparse", meaning that it is composed out of only a few essential objects of interest plus noise/minor non-important objects. Exemplary, seismic data will only show a few essential seismic waves, video data may be composed out of a background and a few moving objects, etc.. Second, we assume that the behavior of these objects follows basic physical concepts. Hence, the change of an object between two observations $\bm{D}_{:k}$ and $\bm{D}_{:(k+1)}$ must be somehow structured. For our model, we focus on the change in position and the change in form of the object. The position of the object in the data describes its movement and should be Lipschitz-continuous where the Lipschitz constant is directly related to the maximum speed that object can physically achieve. The form of an object can, depending on the application, change drastically between two observations, e.g., when there is a sudden change in the camera angle for video recordings, or when the subsurface material changes in seismic measurements. However, we assume that these phenomenons are rare and the total change in form over the data is bounded.

Let us now formulate this idea mathematically. Therefore, we use the following definition.
\begin{definition}\label{def:object}
Let $\bm{U}\in\mathbb{R}^{M\times N}$ and $\bm{\lambda}\in\mathbb{N}^N$. We call a matrix of the form $\mathcal{S}_{\bm{\lambda}}(\bm{U})$ an object. We call
\begin{align*}
\sum\limits_{k=1}^{N-1}\|\bm{U}_{:k}-\bm{U}_{:(k+1)}\|_2^2
\end{align*}
the total change of the object. Furthermore, we say that the object moves Lipschitz-continuous with Lipschitz constant $C$, if $|\bm{\lambda}_k-\bm{\lambda}_{k+1}|\leq C$ for all $k=1,\ldots,N-1$.
\end{definition}
Note that the above definition is for equidistantly sampled data. Otherwise, the total change should be a weighted sum and the Lipschitz-continuity needs to take the distance between two observations into account. Furthermore, we can rewrite the norm as
\begin{align*}
\|\bm{U}_{:k}-\bm{U}_{:(k+1)}\|_2^2=\|\bm{U}_{:k}\|_2^2+\|\bm{U}_{:(k+1)}\|_2^2-2\langle\bm{U}_{:k},\bm{U}_{:(k+1)}\rangle.
\end{align*}
The inner product can also be written as $\langle\bm{U}_{:k},\bm{U}_{:(k+1)}\rangle=\cos\alpha\|\bm{U}_{:k}\|_2\|\bm{U}_{:(k+1)}\|_2$ where $\alpha$ is the angle spanned by both vectors. Hence, minimizing the total change of an object will favor small angles $\alpha$, i.e., so-called (highly) correlated columns.

With Definition \ref{def:object} we can now define the direct and inverse problem of our model. Assume the measured data $\bm{D}$ contains $L$ objects. We can write is as
\begin{align}\label{dataObjectLink}
	\bm{D}=\sum\limits_{l=1}^L\mathcal{S}_{\bm{\lambda^l}}(\bm{U^l})+\text{NOISE},
\end{align}
where $\bm{\lambda^l}$ and $\bm{U^l}$ ($l=1,\ldots,L$) are the $L$ vectors and matrices corresponding to the $L$ objects. Again, the objects represent different parts of the measurement. For video data one object might be a moving car, one might be a pedestrian, and another one represents the background. We assume that $\bm{D}$ is object sparse, which means that the number of objects is small compared to the data size, i.e., $L\ll\min(M,N)$. Indeed this kind of sparsity is similar to the concept of low-rank matrices where the number of non-zero singular values is small. Object sparsity does not require the matrix $\bm{U}$ to be sparse itself, just as low-rank also does not require the singular vectors to be sparse.

The direct problem in our model now formulates as follows. Given the objects $\mathcal{S}_{\bm{\lambda^l}}(\bm{U^l})$ can we simulate the data $\bm{D}$ that would be obtained from the measurements. The answer to this is directly given by Equation (\ref{dataObjectLink}) which we just need to evaluate. Much more complicated is the inverse problem: given the data $\bm{D}$ can we recover the objects. This requires a decomposition of the possibly noisy data $\bm{D}$. To simplify the problem a bit, we use an iterative approach and recover one object at a time. Meaning, our method recovers one object $\mathcal{S}_{\bm{\lambda}}(\bm{U})$ from the given data $\bm{D}$. After this, we form the residual $\bm{R}=\bm{D}-\mathcal{S}_{\bm{\lambda}}(\bm{U})$. The proposed method can then be applied again to the residual to recover the next object. This process can be repeated until all objects are recovered. In each iteration we have to find the object that fits best to the current data, i.e., we need to solve

\begin{align*}
\min\limits_{\bm{U},\bm{\lambda}}\|\bm{D}-\mathcal{S}_{\bm{\lambda}}(\bm{U})\|_F^2\ \ \ \text{s.t. }\begin{matrix}
	\text{total change is bounded}\\
	\bm{\lambda}\text{ is Lipschitz-continuous}
	\end{matrix}.
\end{align*}
This minimization problem has two constraints that need to be handled. The Lipschitz-continuity of $\bm{\lambda}$ is directly related to the speed of the object. We therefore assume that a reasonable bound is known from the application. "How many pixels could the object move from one observation to the next?". We can include this restriction as inequality similar to (\ref{min2}). A-priori information about the total change however might be much harder to obtain. Therefore, we include this constraint using a penalty term as in (\ref{min1}). Moreover, using the properties of the shift-operator (see Definition \ref{def:shiftOperator}) we can rewrite the data fidelity term as
\begin{align*}
\|\bm{D}-\mathcal{S}_{\bm{\lambda}}(\bm{U})\|_F^2
&=
\|\mathcal{S}_{\bm{-\lambda}}\left(\bm{D}-\mathcal{S}_{\bm{\lambda}}(\bm{U})\right)\|_F^2\\
&=
\|\mathcal{S}_{-\bm{\lambda}}(\bm{D})-\bm{U}\|_F^2.
\end{align*}
Altogether, we obtain the minimization problem
\begin{align}
\min\limits_{\bm{U},\bm{\lambda}}\|\mathcal{S}_{-\bm{\lambda}}(\bm{D})-\bm{U}\|_F^2+\mu
\sum\limits_{k=1}^{N-1}\|\bm{U}_{:k}-\bm{U}_{:(k+1)}\|_2^2\label{U_lambda_min_prob}\\
\text{s.t. }|\bm{\lambda}_k-\bm{\lambda}_{k+1}|\leq C\text{ for }k=1,\ldots,N-1\notag
\end{align}
where $\mu$ is a regularization parameter. To get an intuition on what influence the paramter $\mu$ has on the optimal solution, let us consider the two extreme cases $\mu=0$ and $\mu\rightarrow\infty$. For $\mu=0$ the restriction on the matrix $\bm{U}$ vanishes completely and the object can change arbitrarily. This leads to the obvious solution $\bm{U}=\bm{D}$ and $\bm{\lambda}=\bm{0}$ which holds for every constant $C\geq0$. In the other case $\mu\rightarrow\infty$ we force each column of $\bm{U}$ to be equal, i.e., $\bm{U}_{:k}=\bm{u}$ for a vector $\bm{u}$ and all $k=1,\ldots,N$. We now have to minimize $\sum_{k=1}^N\|(\mathcal{S}_{-\bm\lambda}(\bm{D}))_{:k}-\bm{u}\|_2^2$ over $\bm{u}$ and $\bm\lambda$. The optimal solution $\bm{u}$ is the mean over all columns $(\mathcal{S}_{-\bm\lambda}(\bm{D}))_{:k}$ while the optimal shift vector $\bm\lambda$ (still depending on $C$) should minimize the variance over all columns. Note that in this case the optimal solution $\bm{U}=\bm{u}\bm{1}^T$ is a rank-1 matrix. Indeed, we would expect $\bm{U}$ to be approximately of low rank for large parameters $\mu$, i.e., $\bm{U}$ should have only a few large singular values while the rest is small. However, we do not investigate on this conjecture here. If it always holds true and if it can be used to further improve the algorithm remains subject to future work.

Let us now discuss how to solve minimization problem (\ref{U_lambda_min_prob}). Note that the problem is quadratic in $\bm{U}$. Hence, if we would know the shift vector $\bm{\lambda}$, the object matrix $\bm{U}$ would be easy to obtain. In the next paragraphs, we assume $\bm{\lambda}$ to be given and solve the quadratic minimization analytically. We obtain a minimum depending on $\bm{\lambda}$ and can then solve for $\bm{\lambda}$ by minimizing this minimum.

First, we note that the penalty term can be expressed in terms of the rows of the matrix $\bm{U}$ and the matrix $\bm{T}$ by
\begin{align*}
&\sum\limits_{k=1}^{N-1}\|\bm{U}_{:k}-\bm{U}_{:(k+1)}\|_2^2\\
=&\sum\limits_{k=1}^{N-1}\sum\limits_{j=1}^{M}(\bm{U}_{jk}-\bm{U}_{j(k+1)})^2\\
=&\sum\limits_{j=1}^M\sum\limits_{k=1}^{N-1} \bm{U}_{jk}^2-2\bm{U}_{jk}\bm{U}_{j(k+1)}+\bm{U}_{j(k+1)}^2\\
=&\sum\limits_{j=1}^M\bm{U}_{j:}\bm{T}\bm{U}_{j:}^T.
\end{align*}
Next, we divide the data fidelity term into its quadratic, linear, and constant term:
\begin{align*}
&\|\mathcal{S}_{-\bm{\lambda}}(\bm{D})-\bm{U}\|_F^2\\
=&\|\bm{U}\|_F^2-2\langle\mathcal{S}_{-\bm{\lambda}}(\bm{D}),\bm{U}\rangle+\|\mathcal{S}_{-\bm{\lambda}}(\bm{D})\|_F^2\\
=&\|\mathcal{S}_{-\bm{\lambda}}(\bm{D})\|_F^2+\sum\limits_{j=1}^M\bm{U}_{j:}\bm{U}_{j:}^T-2\langle(\mathcal{S}_{\bm\lambda}(\bm{D}))_{j:},\bm{U}_{j:}\rangle.
\end{align*}
The constant term $\|\mathcal{S}_{-\bm{\lambda}}(\bm{D})\|_F^2$ can be ignored for the minimization. Adding both penalty term and data fidelity term back together, we obtain the minimization problem
\begin{align}\label{U_min_prob}
\min\limits_{\bm{U}}\sum\limits_{j=1}^M
\bm{U}_{j:}(\bm{I}+\mu\bm{T})\bm{U}_{j:}^T-2\langle\left(\mathcal{S}_{-\bm{\lambda}}(\bm{D})\right)_{j:},\bm{U}_{j:}\rangle.
\end{align}
Note that these are actually $M$ independent minimization problems for $\bm{U}_{j:}$ with $j=1,\ldots,M$. Each summand is quadratic using the same system matrix $\bm{I}+\mu\bm{T}$. Since $\bm{T}$ is (not strict) diagonal dominant, $\bm{I}+\mu\bm{T}$ is invertible and positive definite independent of the choice of $\mu\geq0$. A quadratic form with a positive definite system matrix is convex and thus its minimum can be found setting the derivative to $\bm{0}$. We obtain the optimal point
\begin{align*}
&&\bm{0}&=2(\bm{I}+\mu\bm{T})\bm{U}_{j:}^T-2\left(\mathcal{S}_{-\bm{\lambda}}(\bm{D})\right)_{j:}^T\\
\Leftrightarrow &&
\bm{U}_{j:}^T&=(\bm{I}+\mu\bm{T})^{-1}\left(\mathcal{S}_{-\bm{\lambda}}(\bm{D})\right)_{j:}^T.
\end{align*}
Plugging the optimal point back into the quadratic form yields an optimal value of
\begin{align*}
&-\left(\mathcal{S}_{-\bm{\lambda}}(\bm{D})\right)_{j:}(\bm{I}+\mu\bm{T})^{-1}\left(\mathcal{S}_{-\bm{\lambda}}(\bm{D})\right)_{j:}^T\\
=&
-\langle
(\bm{I}+\mu\bm{T})^{-1},
\left(\mathcal{S}_{-\bm{\lambda}}(\bm{D})\right)_{j:}^T
\left(\mathcal{S}_{-\bm{\lambda}}(\bm{D})\right)_{j:}
\rangle_F
\end{align*}
for each summand. Adding all terms for $j=1,\ldots,M$ together, we obtain the minimal value of (\ref{U_min_prob})
\begin{align}\label{lambda_min_val}
-\langle
(\bm{I}+\mu\bm{T})^{-1},
\left(\mathcal{S}_{-\bm{\lambda}}(\bm{D})\right)^T
\left(\mathcal{S}_{-\bm{\lambda}}(\bm{D})\right)
\rangle_F.
\end{align}
Extracting the object out of given data can now be done in two steps:
\begin{itemize}
	\item Minimize (\ref{lambda_min_val}) over all Lipschitz-continuous $\bm{\lambda}$.
	\item Solve the quadratic minimization (\ref{U_min_prob}).
\end{itemize}
For the second step, any known quadratic optimization method can be used. The more interesting step is obviously finding $\bm{\lambda}$. Since $\bm{\lambda}\in\mathbb{N}^N$, this is a combinatorial problem. (Note that we use a periodical shift and thus w.l.o.g. $0\leq\bm{\lambda}_k\leq M-1$ for all $k=1,\ldots,N$.) Indeed, the problem is closely related to some maximum clique problems on graphs which are NP-hard, indicating that this problem itself might be NP-hard. In the next section we present an approximation algorithm to recover $\bm{\lambda}$ with exponentially decaying approximation error. Although the algorithm has exponential complexity, we show that the performance can still be handled whenever the object is sufficiently slow or the sampling points are sufficiently close to another (i.e., whenever $C$ is small).

\section{ORKA: algorithm and theory}
In this section we develop an algorithm to maximize the negative of (\ref{lambda_min_val}), which will remove the minus sign at the beginning of the term. We rewrite the term using the explicit form of the Frobenius inner product and obtain
\begin{align}
&\langle
(\bm{I}+\mu\bm{T})^{-1},
\left(\mathcal{S}_{-\bm{\lambda}}(\bm{D})\right)^T
\left(\mathcal{S}_{-\bm{\lambda}}(\bm{D})\right)
\rangle_F\notag\\
=&\sum\limits_{j,k=1}^{N}
\left((\bm{I}+\mu\bm{T})^{-1}\right)_{jk}
\langle\left(\mathcal{S}_{-\bm{\lambda}}(\bm{D})\right)_{:j},\left(\mathcal{S}_{-\bm{\lambda}}(\bm{D})\right)_{:k}\rangle\notag\\
=&\sum\limits_{j,k=1}^{N}
\left((\bm{I}+\mu\bm{T})^{-1}\right)_{jk}
\langle\mathcal{S}_{-\bm{\lambda}_j}(\bm{D}_{:j}),\mathcal{S}_{-\bm{\lambda}_k}(\bm{D}_{:k})\rangle\label{lambda_graph_absolute}
\end{align}
Hence, we need to maximize a weighted sum of inner products where the weights are given by the inverse of the system matrix.

Before we go on and transform the problem further, we want to give a quite simple graph formulation of the above term. This formulation will not be of use later on but is merely for the readers intuition of the problem. We construct the graph in the following way:
\begin{itemize}
	\item The graph has $MN$ vertices, where a vertex is labeled $(j,k)$ with $j=0,\ldots,M-1$, $k=1,\ldots,N$. The vertex $(j,k)$ represents the $k$-th observation shifted by $j$, i.e., it represents the vector $\mathcal{S}_{j}(\bm{D}_{:k})$.
	\item Two vertices $(j,k)$ and $(j',k')$ are connected via an edge if and only if $k\neq k'$. The edges weight is $\left((\bm{I}+\mu\bm{T})^{-1}\right)_{kk'}
	\langle\mathcal{S}_{j}(\bm{D}_{:k}),\mathcal{S}_{j'}(\bm{D}_{:k'})\rangle$.
\end{itemize}
This construction creates a complete $N$-partite graph. Each choice of $\bm{\lambda}$ correlates to a clique in the graph (using the vertices $\{(\bm{\lambda}_k,k)\}_{k=1}^N$. Hence, maximizing (\ref{lambda_graph_absolute}) means solving the generalized maximum clique problem which is NP-hard in general.

Instead, we use a more sophisticated graph construction to find an approximation algorithm for the problem. We make use of the Lipschitz continuity of $\bm{\lambda}$ to reduce the problem size. Therefore, we apply the shift $\mathcal{S}_{\bm{\lambda}_j}$ to both arguments of the inner product in (\ref{lambda_graph_absolute}). By Definition \ref{def:shiftOperator} this does not change its value, the shifts on the left argument cancel out, and the shifts on the right argument can be combined. We obtain
\begin{align}\label{lambda_graph_relative}
\sum\limits_{j,k=1}^{N}
\left((\bm{I}+\mu\bm{T})^{-1}\right)_{jk}
\langle\bm{D}_{:j},\mathcal{S}_{\bm{\lambda}_j-\bm{\lambda}_k}(\bm{D}_{:k})\rangle.
\end{align}
Note that $-C|j-k|\leq \bm{\lambda}_j-\bm{\lambda}_k \leq C|j-k|$. Hence, for each summand we have $2C|j-k|+1$ different possibilities. (Since $\bm{\lambda}\in\mathbb{N}^N$ we can choose $C$ such that $C\in\mathbb{N}$.) Restricting the sum (\ref{lambda_graph_relative}) only to pairs $j,k$ with $|j-k|\leq K$ for $K\ll N$ would reduce the number of possible values drastically. Furthermore, as we will see later on, the matrix $(\bm{I}+\mu\bm{T})^{-1}$ decreases exponentially away from the main diagonal. We define the $K$-approximation as
\begin{align}\label{K_approximation}
	\tau_K(\bm{\lambda})=
	\sum\limits_{\substack{j,k=1\\|j-k|\leq K}}^{N}
	\left((\bm{I}+\mu\bm{T})^{-1}\right)_{jk}
	\langle\bm{D}_{:j},\mathcal{S}_{\bm{\lambda}_j-\bm{\lambda}_k}(\bm{D}_{:k})\rangle.
\end{align}
Note that $\tau_{N-1}(\bm{\lambda})$ is equivalent to (\ref{lambda_graph_relative}). To get bounds on the approximation error, we first need to find an analytic formula for the coefficients of the inverse system matrix. Luckily, the matrix at hand is a near Teoplitz tridiagonal matrix, which have been studied extensively in the past. Formulas for the inverse coefficients can be found e.g., in \cite{Wittenburg98}. For our analysis, we use the form given in \cite{Yueh06}.
\begin{lemma}\label{cor_analyticInverse}
Let $\mu>0$. Then $\bm{I}+\mu\bm{T}$ is invertible and
\begin{align*}
&\left((\bm{I}+\mu\bm{T})^{-1}\right)_{jk}\\
=&\frac{\cosh((N+1-j-k)\phi)+\cosh((N-|j-k|)\phi)}{2\mu\sinh\phi\sinh(N\phi)}\\
\geq&0
\end{align*}
with $\phi=\arccosh(1+\frac{1}{2\mu})\in\mathbb{R}_+$.
\end{lemma}
\begin{proof}
From \cite{Yueh06} (Theorem 2) we get the representation
\begin{align}
&\left((\bm{I}+\mu\bm{T})^{-1}\right)_{jk}\notag\\
=&\frac{\cos((N+1-j-k)\psi)+\cos((N-|j-k|)\psi)}{-2\mu\sin\psi\sin(N\psi)}\label{Yueh_inverse}
\end{align}
with $\cos\psi=1+\frac{1}{2\mu}$. Since $1+\frac{1}{2\mu}>1$, we get a complex argument $\psi$. We use the identity $\cos\psi=\cosh(i\psi)$ and get $\psi=-i\arccosh(1+\frac{1}{2\mu})$. We define $\phi=\arccosh(1+\frac{1}{2\mu})$ and substitute $\psi=-i\phi$. Now using the identities $\cos(-i\phi)=\cosh(\phi)$ and $\sin(-i\phi)=-i\sinh(\phi)$ we can replace all trigonometric functions with hyperbolic functions in (\ref{Yueh_inverse}). The minus in the denominator cancels with $(-i)^2$ and we obtain the result.
\end{proof}
We can already see that the inverse of the system matrix decreases rapidly away from the main diagonal and towards the main anti-diagonal.

The representation of the inverse matrix given in Lemma \ref{cor_analyticInverse} is useful for our theoretical analysis. However, when it comes to the numerical implementation of the inverse matrix (which we need for the ORKA algorithm), we do not suggest to use this formula. Both, numerator and denominator can become very large for large $N$ because $\cosh$ and $\sinh$ grow exponentially. The entries of the inverse matrix on the contrary are usually quite small. This contrast can lead to a loss in precision when using the formula. A numerical better representation is given in the following Lemma.
\begin{lemma}\label{coeff_form_sum}
The eigenvalues $\lambda_l$, $l=0,\ldots,N-1$ of $\bm{T}$ are
\begin{align*}
\lambda_l=4\sin\left(\frac{l\pi}{2N}\right)^2.
\end{align*}
The according eigenvectors $\bm{x}^l$ are
\begin{align*}
\bm{x}^l_j=\begin{cases}
\sqrt{\frac{1}{N}} & l=0 \\
\sqrt{\frac{2}{N}}\cos\left(\frac{l(j-1/2)\pi}{N}\right) & l\neq0
\end{cases}.
\end{align*}
Hence, the inverse of $\bm{I}+\mu\bm{T}$ can be written as
\begin{align*}
&\left((\bm{I}+\mu\bm{T})^{-1}\right)_{jk}\\
=&\frac{1}{N}+\frac{2}{N}\sum\limits_{l=1}^{N-1}\frac{\cos\left(\frac{l(j-1/2)\pi}{N}\right)\cos\left(\frac{l(k-1/2)\pi}{N}\right)}{1+4\mu\sin\left(\frac{l\pi}{2N}\right)^2}\\
=&\frac{1}{N}+\frac{1}{N}\sum\limits_{l=1}^{N-1}\frac{\cos\left(\frac{l(j-k)\pi}{N}\right)+\cos\left(\frac{l(j+k-1)\pi}{N}\right)}{1+4\mu\sin\left(\frac{l\pi}{2N}\right)^2}\\
\end{align*}
\end{lemma}
\begin{proof}
The eigendecomposition of $\bm{T}$ can be found e.g., in \cite{Losonczi92}. The representation of the inverse matrix is obtained by using the eigendecomposition with reciprocal eigenvalues. Last, we use product-to-sum identity for cosine.
\end{proof}
Note that the last equation gives a form of the inverse in terms of its diagonal and anti-diagonal parts similar to Lemma \ref{cor_analyticInverse}. We can now prove some approximation results for our $K$-approximation.
\begin{corollary}\label{cor_approx}
Let $\|\bm{D}_{:k}\|_2\leq1$ for all $k=1,\ldots,N$. Then for all $\bm{\lambda}\in\mathbb{N}^N$ and $K<N-1$ we have
\begin{align}\label{upper_bound}
|\tau_{N-1}(\bm{\lambda})-\tau_K(\bm{\lambda})|\leq G(\mu,N)(N-K)^2e^{(N-K)\phi}
\end{align}
with $\phi$ as in Lemma \ref{cor_analyticInverse} and
\begin{align*}
0\leq G(\mu,N)=\frac{1}{\mu\sinh\phi\sinh(N\phi)}\leq\frac{1}{N},
\end{align*}
\begin{align*}
\lim\limits_{\mu\rightarrow0}G(\mu,N)=0,
&&
\lim\limits_{\mu\rightarrow\infty}G(\mu,N)=\frac{1}{N}.
\end{align*}
Furthermore, there exist matrices such that
\begin{align}\label{lower_bound}
	|\tau_{N-1}(\bm{\lambda})-\tau_K(\bm{\lambda})|\geq \frac{1}{8}G(\mu,N)(N-K)^2e^{(N-K)\phi/2}.
\end{align}
\end{corollary}
\begin{proof}
By Lemma \ref{cor_analyticInverse} all entries of the inverse matrix are positive. We can use Cauchy-Schwarz inequality on the inner product and obtain
\begin{align*}
&|\tau_{N-1}(\bm{\lambda})-\tau_K(\bm{\lambda})|\\
=&
\left|\sum\limits_{\substack{j,k=1\\|j-k|> K}}^{N}
\left((\bm{I}+\mu\bm{T})^{-1}\right)_{jk}
\langle\bm{D}_{:j},\mathcal{S}_{\bm{\lambda}_j-\bm{\lambda}_k}(\bm{D}_{:k})\rangle\right|\\
\leq&\sum\limits_{\substack{j,k=1\\|j-k|> K}}^{N}
\left((\bm{I}+\mu\bm{T})^{-1}\right)_{jk}
\end{align*}
This inequality is sharp for e.g., $\bm{D}_{jk}=1$. Hence, the (\ref{upper_bound}) and (\ref{lower_bound}) are true if we can bound the sum of inverse matrix coefficients from above and below. We use the coefficient representation of Lemma \ref{cor_analyticInverse}. Ignoring the index independent denominator for now, we give bounds for
\begin{align}\label{coeff_sum}
\sum\limits_{\substack{j,k=1\\|j-k|> K}}^{N}\cosh((N+1-j-k)\phi)+\cosh((N-|j-k|)\phi)
\end{align}
The sum has $(N-K)(N-K-1)/2$ terms. The largest summand is obtained at $j=1$, $k=K+2$ (or $k=1$, $j=K+2$). Thus, (\ref{coeff_sum}) is bounded from above by
\begin{align}
&\frac{1}{2}(N-K)(N-K-1)\notag\\
&\cdot\left(\cosh((N-K-2)\phi)+\cosh((N-K-1)\phi)\right)\notag\\
\leq&(N-K)^2\cosh((N-K)\phi).\label{cosh_inequality}
\end{align}
Now we use that $\cosh(x)\leq2e^x$ for $x\geq0$ and obtain (\ref{upper_bound}) with $G=1/(\mu\sinh\mu\sinh(N\mu))$.

For a lower bound of (\ref{coeff_sum}) we use that $\cosh$ is a convex function. Dividing (\ref{coeff_sum}) by $(N-K)(N-K-1)$ we can interpret it as mean of function values By convexity this can be bounded from below by applying the function to the mean sample point, which is
\begin{align*}
&\phi\sum\limits_{\substack{j,k=1\\|j-k|> K}}^{N}\frac{(N+1-j-k)+(N-|j-k|)}{(N-K)(N-K-1)}\\
=&\frac{\phi}{3}(2N-2K-1).
\end{align*}
Therefore, (\ref{coeff_sum}) is bounded from below by
\begin{align*}
(N-K)(N-K-1)\cosh\left(\frac{(2N-2K-1)\phi}{3}\right).
\end{align*}
Since $N-K\geq2$, we have $(N-K)(N-K-1)\geq(N-K)^2/2$ and $2N-2K-1\geq1.5(N-K)$. We use $\cosh(x)\geq\frac{1}{2}e^x$ for $x\geq0$ and get the lower bound
\begin{align*}
\frac{1}{4}(N-K)^2e^{(N-K)\phi/2}.
\end{align*}
From this (\ref{lower_bound}) directly follows. It remains to show the bounds and limits for $G(\mu,N)$. We show that it is monotonically increasing in $\mu$ and prove the limits. The inequalities then follow as a consequence.

Since $\phi$ and $\mu$ are not independent, we first need to rewrite $G$ in terms of only one of these variables. In this case the analysis is much easier when we rewrite $\mu$ as $\mu=\frac{1}{2}(\cosh(\phi)-1)^{-1}$. We obtain
\begin{align*}
G(\phi,N)=2\frac{\cosh(\phi)-1}{\sinh(\phi)\sinh(N\phi)}.
\end{align*}
Since $\phi$ is monotonically decreasing in $\mu$, $G(\mu,N)$ is increasing if and only if $G(\phi,N)$ is decreasing. For this, we analyze the derivative of $G(\phi,N)$ in $\phi$. The denominator of the derivative will be the squared denominator of $G(\phi,N)$ and thus always positive. Hence, it is enough to show that the numerator is negative to prove that $G(\phi,N)$ is decreasing. Calculating the numerator of the derivative and simplifying we obtain
\begin{align}\label{second_term}
2(\cosh(\phi)-1)(\sinh(N\phi)-N\sinh(\phi)\cosh(N\phi)).
\end{align}
Note that $\cosh(\phi)-1$ is always positive. Thus we need to show that the second term is always negative. For $\phi=0$ the second term becomes $0$. The derivative of the second term is
\begin{align*}
N\cosh(N\phi)(1-\cosh(\phi))-N^2\sinh(\phi)\sinh(N\phi),
\end{align*}
which is all negative. Hence, the second term of (\ref{second_term}) starts at $0$ and is decreasing, i.e., it is always negative. Thus $G(\phi,N)$ is decreasing and finally $G(\mu,N)$ is monotonically increasing in $\mu$.

For the limits we simply use Lemma \ref{coeff_form_sum} to see that
\begin{align*}
\lim\limits_{\mu\rightarrow\infty}\left((\bm{I}+\mu\bm{T})^{-1}\right)_{jk}=\frac{1}{N}.
\end{align*}
Furthermore, from the continuity of matrix inversion, we also know that
\begin{align*}
\lim\limits_{\mu\rightarrow0}\left((\bm{I}+\mu\bm{T})^{-1}\right)_{jk}=\begin{cases}
1 & j=k\\ 0 & j\neq k
\end{cases}.
\end{align*}
Now using Lemma \ref{cor_analyticInverse} again, we can rewrite $G(\mu,N)$ as
\begin{align*}
G(\mu,N)=\frac
{2\left((\bm{I}+\mu\bm{T})^{-1}\right)_{jk}}
{\cosh((N+1-j-k)\phi)+\cosh((N-|j-k|)\phi)}
\end{align*}
for all $j,k$. Note that $\lim\limits_{\mu\rightarrow0}\phi=\infty$ and $\lim\limits_{\mu\rightarrow\infty}\phi=0$. The limits for $G(\mu,N)$ now directly follow by applying the limits to the right-hand side.
\end{proof}
Corollary \ref{cor_approx} gives an approximation error bound with one exponentially decreasing and one quadratically decreasing part. It also shows that for some matrices this error bound is achieved up to some small constants. For small $\phi$ (large regularization parameter $\mu$) the exponential decay rate can become slow. Hence, we give the following error bound that is more useful in this case.
\begin{corollary}\label{cor_approx2}
Let $\|\bm{D}_{:k}\|_2\leq1$ for all $k=1,\ldots,N$. Then for all $\bm{\lambda}\in\mathbb{N}^N$ and $K<N-1$ we have
\begin{align*}
	&|\tau_{N-1}(\bm{\lambda})-\tau_K(\bm{\lambda})|\\
	\leq& \frac{1}{2}G(\mu,N)(N-K)^2e^{(N-K)^2\phi^2/2}.
\end{align*}
\end{corollary}
\begin{proof}
Instead of $\cosh(x)\leq2e^x$ we use $\cosh(x)\leq e^{x^2/2}$ on (\ref{cosh_inequality}).
\end{proof}

Now, we define the best $C$-Lipschitz continuous $K$-approximation as
\begin{align}\label{best_K_approx}
\bm{\lambda}^\text{max}_K=\arg\max\limits_{\substack{\bm{\lambda}\in\mathbb{N}^N\\|\bm{\lambda}_k-\bm{\lambda}_{k+1}|\leq C}}\tau_K(\bm{\lambda}).
\end{align}
Note that $-\tau_{N-1}(\bm{\lambda}^\text{max}_{N-1})$ is by definition the optimal value of (\ref{U_lambda_min_prob}). We will use the $K$-Approximation to find the shift vector, but the quadratic system (\ref{U_lambda_min_prob}) will be solved exactly. Thus, our algorithm will return the value $-\tau_{N-1}(\bm{\lambda}^\text{max}_{K})$. We now get the following error bounds directly following from Corollary \ref{cor_approx} and \ref{cor_approx2}.
\begin{theorem}\label{thm_errorBounds}
Let $\|\bm{D}_{:k}\|_2\leq1$ for all $k=1,\ldots,N$. Then for $K<N-1$ we have
\begin{align*}
&\tau_{N-1}(\bm{\lambda}^\text{max}_{N-1})-\tau_{N-1}(\bm{\lambda}^\text{max}_{K})\\
\leq&G(\mu,N)(N-K)^2\begin{cases}
2e^{(N-K)\phi} \\
e^{(N-K)^2\phi^2/2}
\end{cases}
\end{align*}
where both bounds hold but depending on the regularization $\mu$ one might be preferable over the other.
\end{theorem}
\begin{proof}
Let $E$ be the error bound obtained either by Corollary \ref{cor_approx} or \ref{cor_approx2}. Then we have
\begin{align*}
&\tau_{N-1}(\bm{\lambda}^\text{max}_{N-1})-\tau_{N-1}(\bm{\lambda}^\text{max}_{K})\\
\leq&\tau_K(\bm{\lambda}^\text{max}_{N-1})-\tau_{N-1}(\bm{\lambda}^\text{max}_{K})+E\\
\leq&\tau_K(\bm{\lambda}^\text{max}_{K})-\tau_{N-1}(\bm{\lambda}^\text{max}_{K})+E\\
\leq& 2E
\end{align*}
\end{proof}
\begin{remark}
The requirement $\|\bm{D}_{:k}\|_2\leq1$ is just set to simplify the notation a bit. A closer look at the proof of Corollary \ref{cor_approx} shows that any bound on $\left|\langle\bm{D}_{:j},\mathcal{S}_{\bm{\lambda}_j-\bm{\lambda}_k}(\bm{D}_{:k})\rangle\right|$ would just carry over as a constant into the error bound. A quite simple (but also rough) bound is $\|\bm{D}\|_F^2/2$.
\end{remark}

Now that we have proven that we can approximate (\ref{U_lambda_min_prob}) by solving (\ref{best_K_approx}), we can finally discuss the algorithm itself. The idea is, to simply construct a graph on which the optimal solution coincides with the longest path between two points on the graph. We can then use standard graph algorithms to solve the problem. The graph construction is given in the following definition.

\begin{definition}
Let $K<N-1$ and the Lipschitz constant $C\in\mathbb{N}$ be given. We construct the $(K,C)$-approximation graph as follows.

For all integers $-C\leq r_j\leq C$ with $j=1,\ldots,N$ we create vertices labeled $(j,r_{j-1},\ldots,r_{\max(1,j-K+1)})$. Furthermore, we create one vertex labeled $(N+1)$.

Now, we add a directed edges labeled as $(j,r_{j-1},\ldots,r_{\max(1,j-K)})$ going from vertex $$(j-1,r_{j-2},\ldots,r_{\max(1,j-K)})$$ to vertex $$(j,r_{j-1},\ldots,r_{\max(1,j-K+1)}).$$ Furthermore, we add $0$-weight edges from all vertices $(N,\ldots)$ to $(N+1)$. The weight of all other edges $(j,r_{j-1},\ldots,r_{\max(1,j-K)})$ is given by
\begin{align*}
\sum\limits_{k=j-1}^{\max(1,j-K)}
\left((\bm{I}+\mu\bm{T})^{-1}\right)_{jk}
\langle\bm{D}_{:j},\mathcal{S}_{s_{k,j-1}}(\bm{D}_{:k})\rangle
\end{align*}
with
\begin{align*}
s_{k,j-1}=\sum\limits_{l=k}^{j-1}r_l.
\end{align*}
\end{definition}
The intuition behind this graph is as follows. Let $\bm{\lambda}$ be any Lipschitz continuous shift vector. Define $\tilde{r}_j=\bm{\lambda}_j-\bm{\lambda}_{j+1}$ for $j=1,\ldots,N-1$. Summing up the weights of the edges $(j,\tilde{r}_{j-1},\ldots,\tilde{r}_{\max(1,j-K)})$ for $j=1,\ldots,N$ yields
\begin{align*}
\sum\limits_{\substack{j,k=1\\|j-k|\leq K\\j> K}}^{N}
\left((\bm{I}+\mu\bm{T})^{-1}\right)_{jk}
\langle\bm{D}_{:j},\mathcal{S}_{\bm{\lambda}_j-\bm{\lambda}_k}(\bm{D}_{:k})\rangle
\end{align*}
which is $\tau_K(\bm{\lambda})/2$ up to a constant term. (See (\ref{K_approximation}), use that all involved terms are symmetric and the terms are independent of $\bm{\lambda}$ for $j=k$.) Hence each shift vector describes a path from vertex $(1)$ to $(N+1)$ (adding a $0$-weight edge at the end). The reverse also holds. Any path in the graph from $(1)$ to $(N+1)$ can be associated to a shift vector. The directed edges are constructed in a way, that a path has to pass through $N+1$ vertices of the form $(1)$, $(2,r_1)$, $(3,r_2,r_1)$, $\ldots$, $(N,\ldots)$, $(N+1)$. The edges also ensure that the relative shift $\bm{\lambda}_j-\bm{\lambda}_{j+1}$ is consistent throughout the connected vertices. Hence, finding the best $K$-approximation (\ref{best_K_approx}) is equivalent to finding the path with the largest weight sum from $(1)$ to $(N+1)$. This problem is also known as longest path problem and is NP-hard in general but luckily can be solved in linear time for directed acyclic graphs which the $(K,C)$-approximation graph is. Altogether, we obtain the following algorithm for object reconstruction using K-approximation (ORKA).

\begin{algorithm}
	\KwData{data $\bm{D}\in\mathbb{R}^{M\times N}$, parameters $\mu\geq0$, $C,K\in\mathbb{N}$}
	\KwResult{object $\mathcal{S}_{\bm{\lambda}}(\bm{U})$}
	Construct the $(K,C)$-approximation graph\;
	Find longest path from vertex $(1)$ to $(N+1)$\;
	Reconstruct $\bm{\lambda}$ from the path\;
	Solve the quadratic optimization (\ref{U_min_prob}) for $\bm{U}$\;
	\caption{ORKA - object reconstruction using {$K$-approximation}}
\end{algorithm}

\begin{remark}
The $(K,C)$-approximation graph has $O(N(2C+1)^{K-1})$ vertices and $O(N(2C+1)^K)$ edges, thus it scales exponentially. If the sampling rate of the data is not too low (compared to the object speed), then $C$ will be small. $K$ can then be chosen accordingly to balance approximation error and computational complexity. Moreover, the required memory can be reduced further when realizing that the graph is indeed $(N+1)$-partite and a good implementation would only require one partition at a time to be in memory.

Furthermore, note that $\bm{\lambda}$ can only be reconstructed up to a constant from its relative shifts $\bm{\lambda}_j-\bm{\lambda}_{j+1}$, i.e., $\bm{\lambda}+t$ will have the same differences for any $t\in\mathbb{N}$. We assume here that $\bm{\lambda}_1=0$. This eliminates ambiguities of the original problem (\ref{U_lambda_min_prob}) where each $\bm{\lambda}+t$ would result in the same optimal value.
\end{remark}

\begin{remark}
We can change optimization problem (\ref{U_lambda_min_prob}) by considering non-equidistant sampling, or extending the penalty term to consider more than one neighboring column. The presented algorithm can still be applied as long as the resulting system matrix is invertible. Error bounds can be obtained in a similar manner as long as values (or bounds) for the entries of the inverse system matrix are known.
\end{remark}

\section{Numeric}

In this section we present some numerical experiments to demonstrate the ORKA algorithm. In the first subsection we analyze the runtime and approximation error for different parameter setups. In the second subsection we apply the algorithm to seismic and ultrasonic data for denoising and object reconstruction. Finally, in the last subsection we show how the method can be applied to 3D video data.

Note that the proposed method is not suitable for sparse data approximation. In this problem one seeks for an efficient representation (or approximation) of given data, i.e., a representation that is as exact as possible while having a small amount of storage cost. For example, Wavelets or Curvelets are used to achieve this. Prominent formats such as mp3 or jpg try to solve this exact problem. The ORKA algorithm however does not fit in this setup for two reasons. First, the storage cost for one object ($\bm{\lambda}$ and $\bm{U}$) are already higher than just storing the data $\bm{D}$ itself. Second, as we have seen in the discussion after (\ref{U_lambda_min_prob}), we can simply choose $\mu=0$ to obtain perfect approximation. Instead, ORKA should be used with parameters $C,\mu$ that suit the application. We give examples on how this can be done in the following experiments.

\subsection{Error and Runtime}

We analyze the runtime and approximation error of the proposed method. The runtime experiment was performed on a 6 core i7-8700 3.2Ghz with 64GB memory using MatLab R2020a where part of the code was coded in a C++ mex-file that is able to run the code multi-threaded (i.e., parallel). We only measured the runtime of $K$-approximation (i.e., recovering $\bm{\lambda}$). The complexity of solving the quadratic system can vary depending on the preferred method. First, we measure the runtime for increasing parameter $K$ while $C=1$ and $M=N=64$ are fixed. We would expect a complexity of $O((2C+1)^K)$ which is the number of edges per partition in the $(K,C)$-approximation graph. In Figure \ref{fig:runtime_K} we show the runtime for the parallel and serial (non-parallel) implementation. A third line that is exactly $O((2C+1)^K)$ is given as reference. Note that the y-axis is log-scaled. We can see that both implementations behave exactly as expected. The parallel implementation has a larger overhead at the beginning but is more efficient for larger $K$.

Next, we analyze the runtime for increasing matrix size $\bm{D}$. We use a squared matrix $\bm{D}\in\mathbb{R}^{N\times N}$ and measure the runtime against increasing $N$. The number of partitions of the $(K,C)$-approximation graph increases linearly in $N$, so we would expect a linear increase in runtime from this. However, we also need to calculate the required inner products between different column shifts. This can either be done by calculating all inner products using Fourier transform ($O(K^2N\log N)$) and then picking the required values, or by calculating only the required shift combinations ($O(K^3(2C+1)N)$). We are using the implementation using Fourier transform here. Figure \ref{fig:runtime_N} shows the obtained runtime. As we can see, the increase is indeed nearly linear. We use the non-parallel implementation here to avoid the overhead to distort the measurements. The efficiency of the Fast Fourier Transform relies on the data length $N$ being a multiple of $2$ or at least having only small prime factors. Hence, we observe spikes in the runtime, whenever very large prime factors appear in $N$. As reference, we have plotted the largest prime of $N$ on the bottom of the graph (use the right y-axis).

\begin{figure}
\centering
\subfloat[]{
\includegraphics[height=.4\columnwidth]{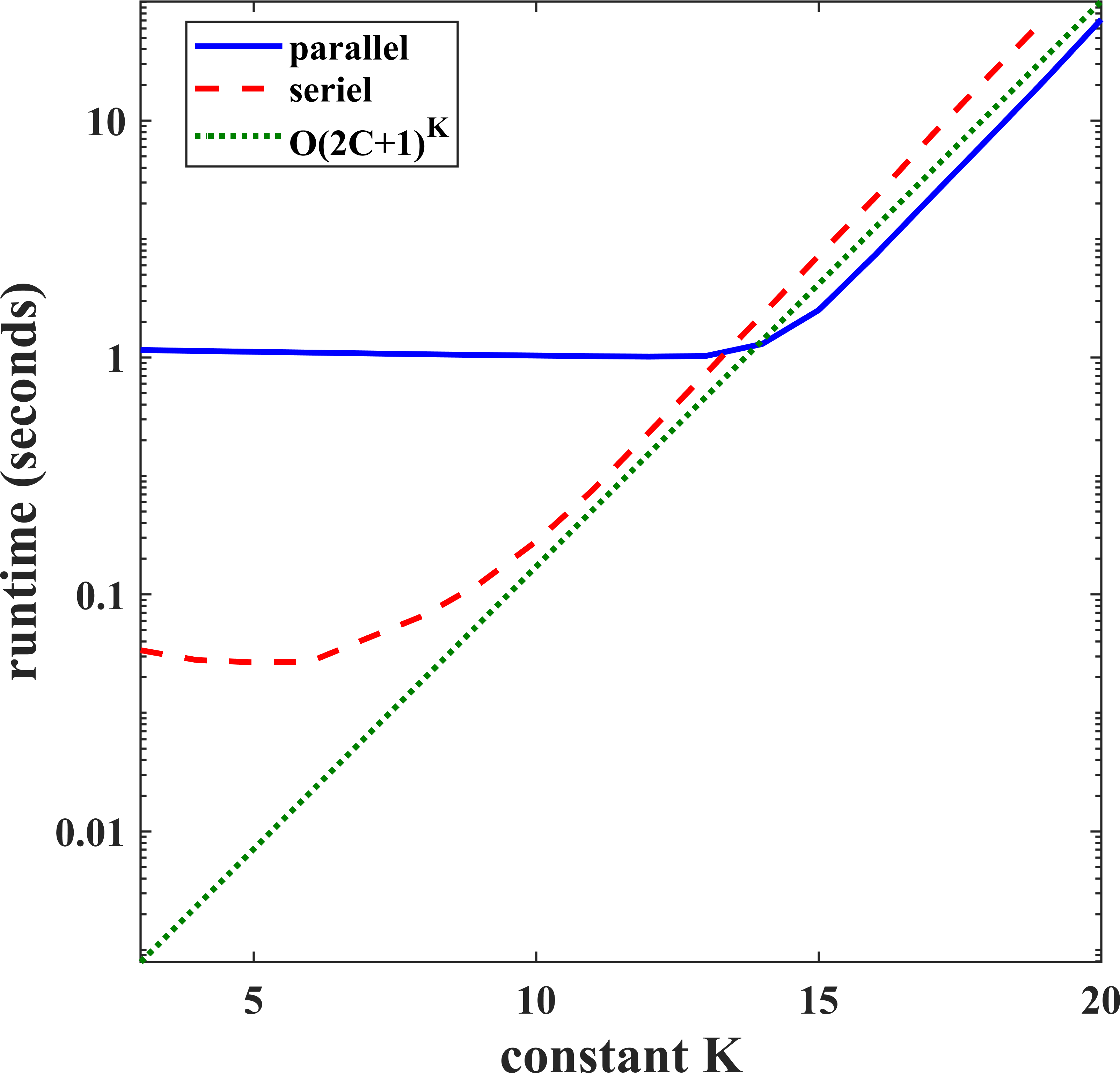}
\label{fig:runtime_K}}
\subfloat[]{
\includegraphics[height=.4\columnwidth]{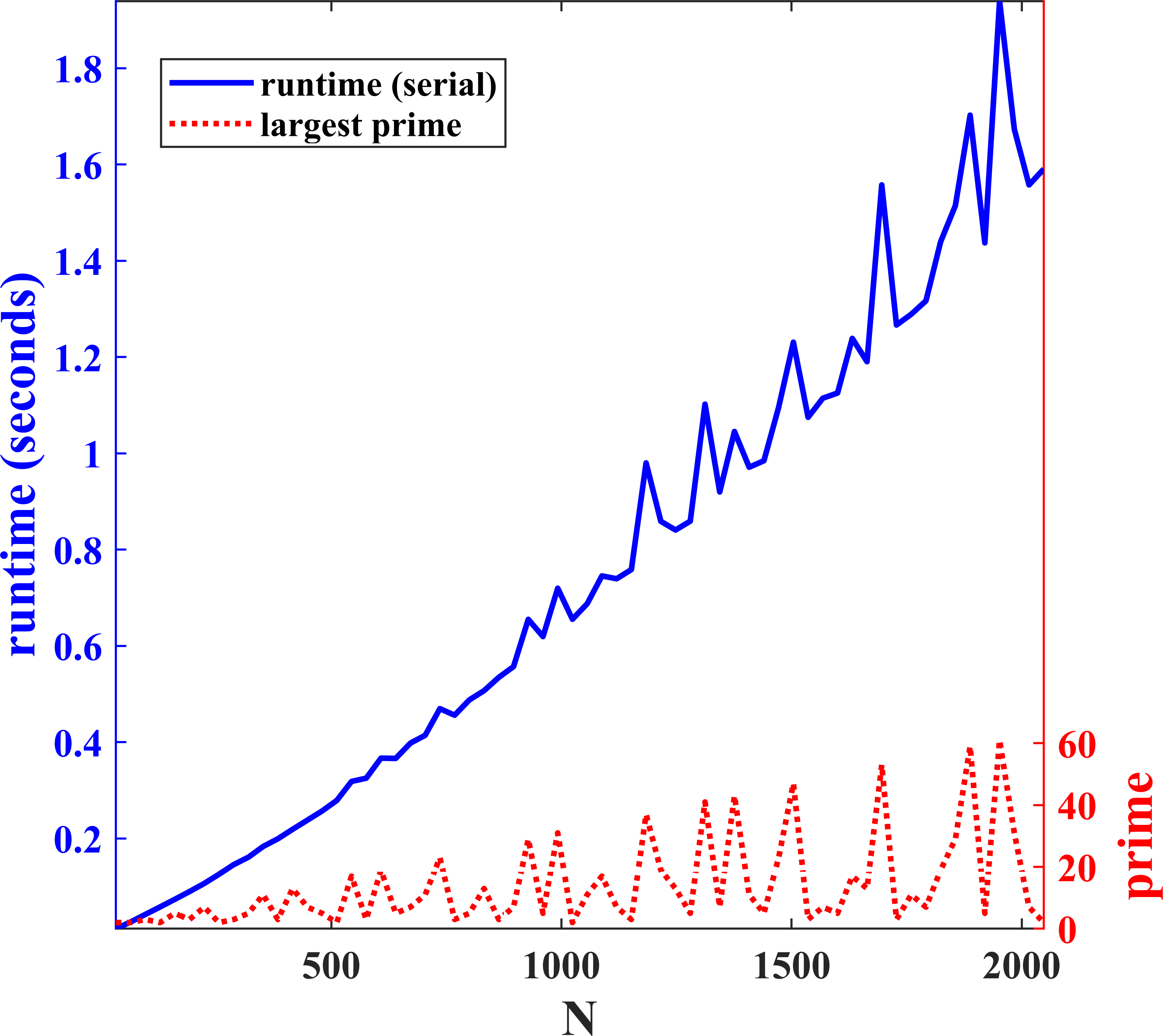}
\label{fig:runtime_N}}
\caption{Runtime of $K$-approximation depending on (a) $K$ and (b) matrix size $N\times N$.}
\end{figure}

Next, we analyze the behavior of $K$-approximation. In a first experiment construct the matrix
\begin{align*}
	\bm{D}=\text{diag}([1,1,0,1,0,0,1,0,0,0,1,\ldots])
\end{align*}
i.e., a diagonal matrix with ones on the diagonal that are separated by an increasing number of zeros. We set $C=1$ allowing a relative shift of $\pm1$ for neighboring columns and $\mu=1000$ to force highly correlated columns. Then, the best shift is $\bm{\lambda}=[0,1,2,\ldots]$, which shifts all non-zero entries into the first row. However, $K$-approximation only compares columns which are at most $K$ columns away from each other. Thus, if the $0$-caps in $\bm{D}$ become too large, $K$-approximation can no longer see the optimal path. We use $\bm{D}$ with a largest gap of $14$ zeros (see Figure \ref{fig:gap_mtx}) and run $K$-approximation for $K=3,\ldots,15$. Figure \ref{fig:gap_path} shows the reconstructed $\bm{\lambda}$. Note that we shifted all $\bm{\lambda}$ each by a different constant value to increase visibility of the plot. We see that only for $K=15$ the perfect shift vector is found. In all other cases the reconstruction gets lost when the gaps get too large and the algorithm has to choose a random direction. (In our implementation always the smallest shift is chosen, whenever the optimal shift is not unique. In this case the smallest is $-C=-1$.) Hence, if the object one wants to reconstruct cannot be found in every measurement (e.g., because the measurement failed or the object is covered by other signals), then $K$ needs to be chosen large enough to cover the gap.
	
\begin{figure}
\centering
\subfloat[]{
\includegraphics[height=.4\columnwidth]{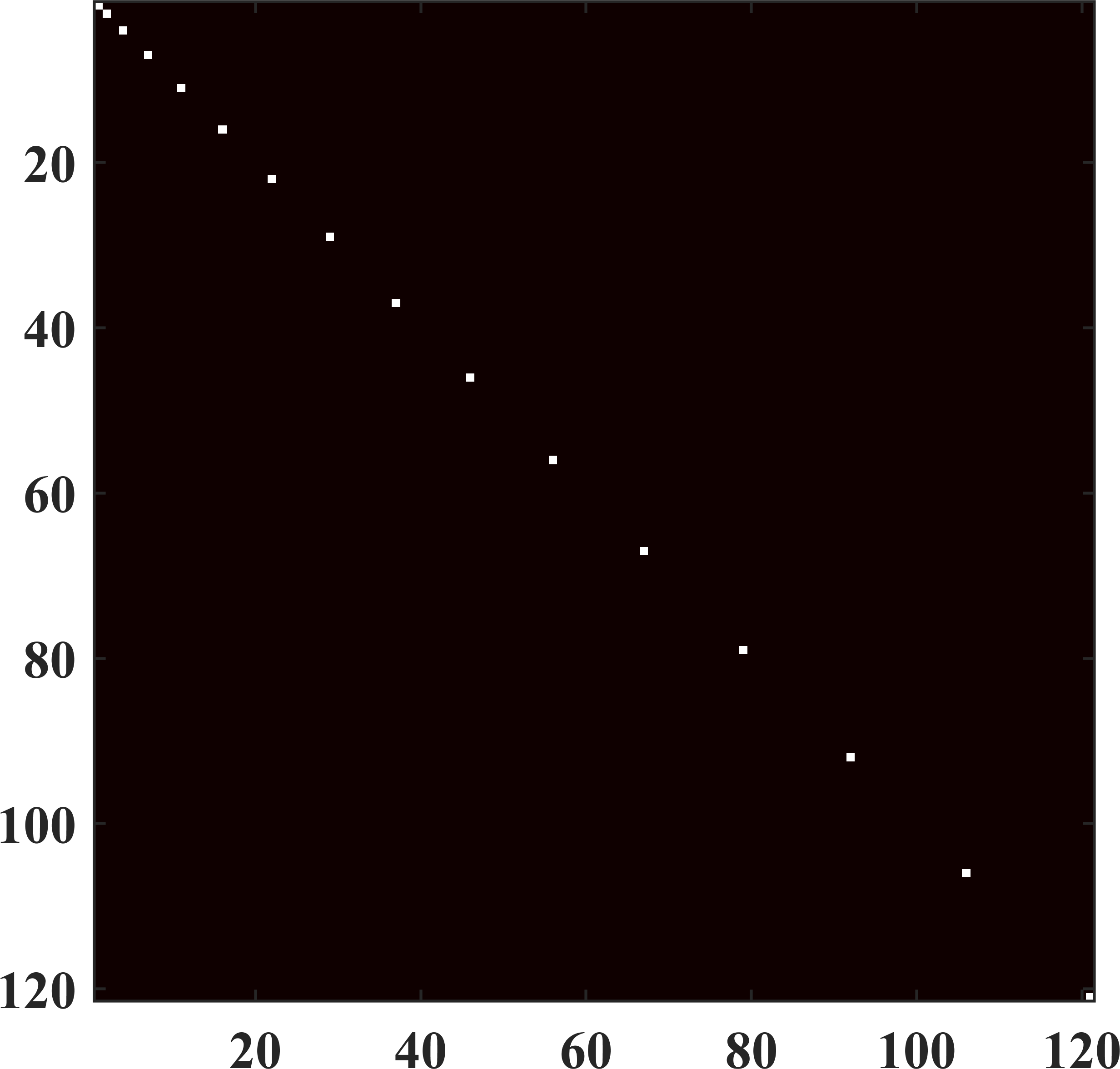}
\label{fig:gap_mtx}}
\subfloat[]{
\includegraphics[height=.4\columnwidth]{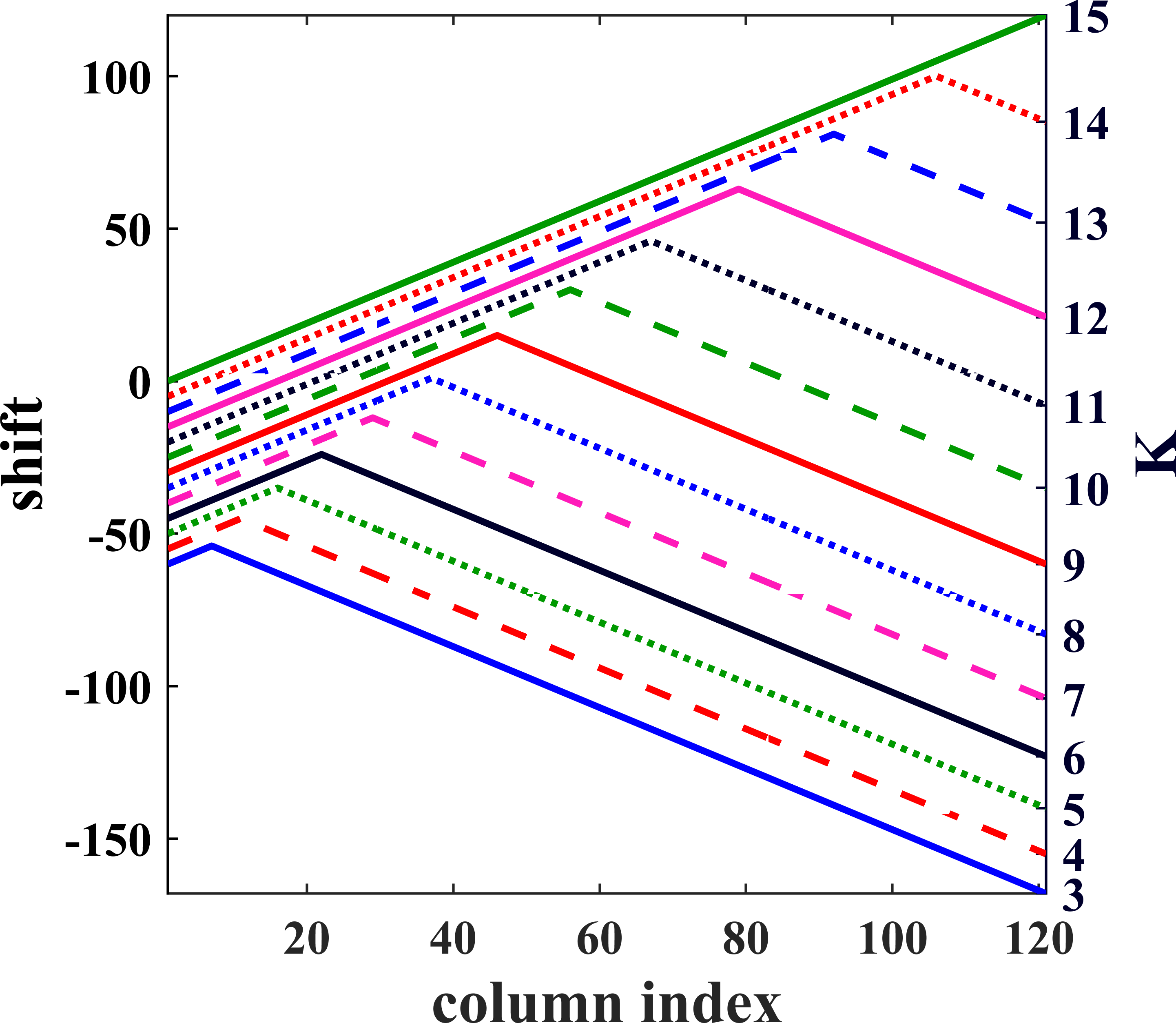}
\label{fig:gap_path}}
\caption{(a) Diagonal matrix with $0$-gaps and (b) the reconstructed shift vectors $\bm{\lambda}$ for different $K$.}
\end{figure}

Let us now test the approximation error depending on the parameter $K$. Numerically testing the approximation error is troublesome as we need to know the optimal solution $\tau_{N-1}(\bm{\lambda}^\text{max}_{N-1})$. Hence, for this experiment we restrict ourselves to small matrices $\bm{D}\in\mathbb{R}^{64\times16}$ and C=1. Under this setup, the best approximating object $\mathcal{S}_{\bm\lambda}(\bm{U})$ can still be calculated by going through all possible combinations for $\bm\lambda$. Indeed, for $K=16$ ORKA will return this global optimum without approximation error. We ran ORKA with $K=3,\ldots,16$ and different parameter $\mu$ on $100$ random matrices $\bm{D}$ (with normalized columns). The mean approximation error compared to the optimal solution over all $100$ runs is shown in Figure \ref{fig:KvsOpt}. We can see that the approximation error decreases rapidly with $K$ but larger parameter $\mu$ lead to a larger approximation error. This coincides with Theorem \ref{thm_errorBounds}.

\begin{figure}
\centering
\includegraphics[height=.4\columnwidth]{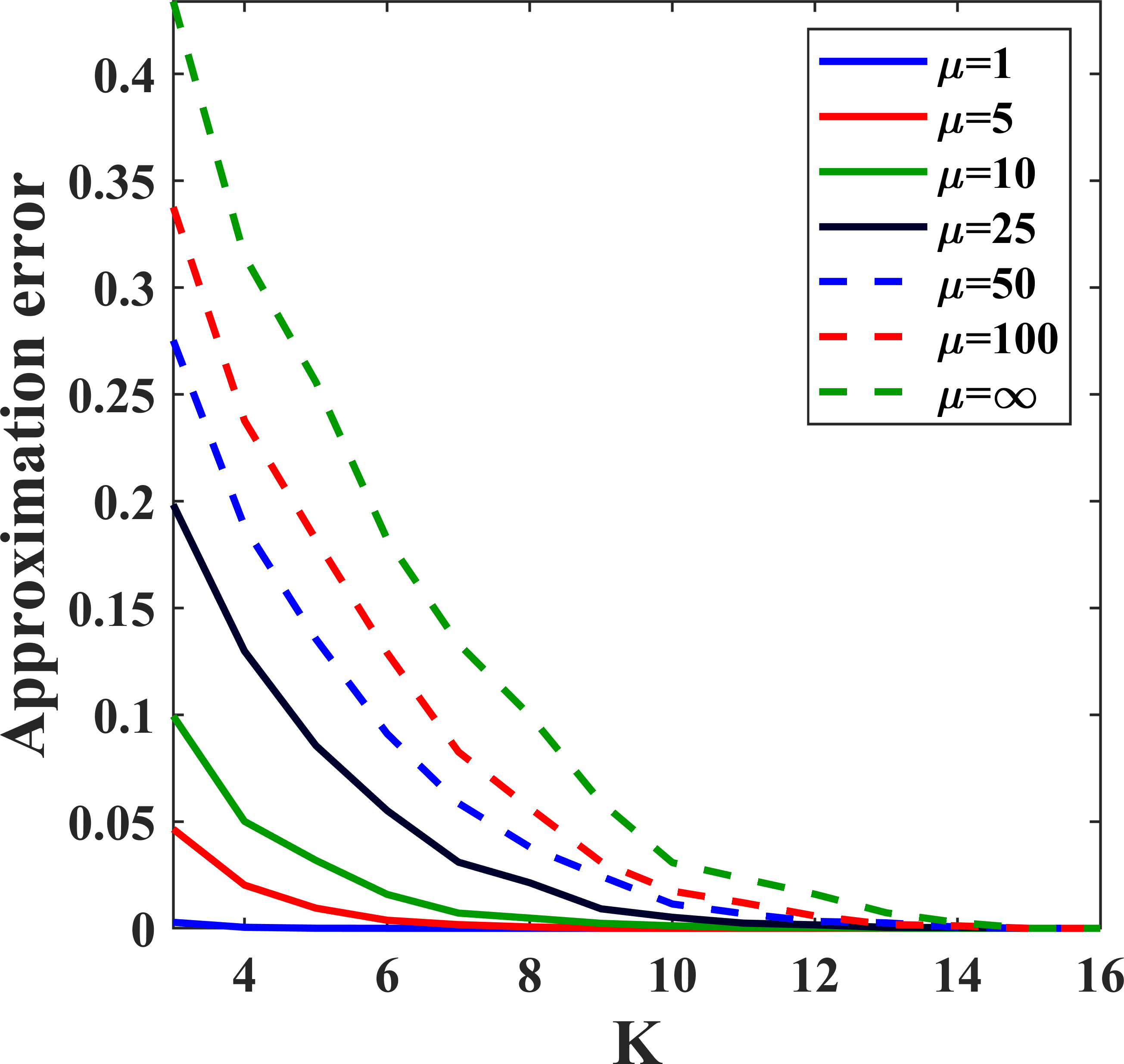}
\caption{Approximation error of ORKA against $K$ for different parameters $\mu$. For $K=16$ the global optimum of (\ref{U_lambda_min_prob}) is obtained, i.e., no approximation error.}
\label{fig:KvsOpt}
\end{figure}

\subsection{Seismic and ultrasonic data}

In this section we perform tests on $2D$ real data from two different applications. Figure \ref{fig:realData} shows the used data $\bm{D}$ from ultrasonic non-destructive testing (Figure \ref{fig:realData_ndt}) and seismic exploration (Figure \ref{fig:realData_seis}). Both applications fit the proposed model for object sparse data. In this case, the objects are (seismic or ultrasonic) waves that is measured from different probes/stations. The ultrasonic data was obtained from a steel tube test where emitter and receivers are placed on the outside surface of the tube. The data contains three relevant ultrasonic waves (see the markings in Figure \ref{fig:realData_ndt}). The first wave is a wave that travels directly along the surface from emitter to receiver. The second wave is a reflection from a defect somewhere inside the tube. The last wave is the reflection from the inner surface. It is split into two reflections on the left side of the data what could indicate a defect in the inner surface. The task is, to automatically separate the three signals and gain information about the position and type of the defect. Therefore, we want to use ORKA to extract the different waves. The obtained objects $\mathcal{S}_{\bm\lambda}(\bm{U})$ already give information about the form and position of these waves. If a more sophisticated analysis of the defect is required, it can now be done on each wave separately to avoid interference between the waves. The seismic data (Figure \ref{fig:realData_seis}) has a similar setup. Instead of steel tubes, the subsurface is analyzed using seismic waves. Boundaries of earth layers cause reflections of these waves that are then recorded at different positions on the surface. Each of these reflected waves forms one object of the data and corresponds to one subsurface earth layer.

\begin{figure}
\centering
\subfloat[]{
\includegraphics[height=.4\columnwidth]{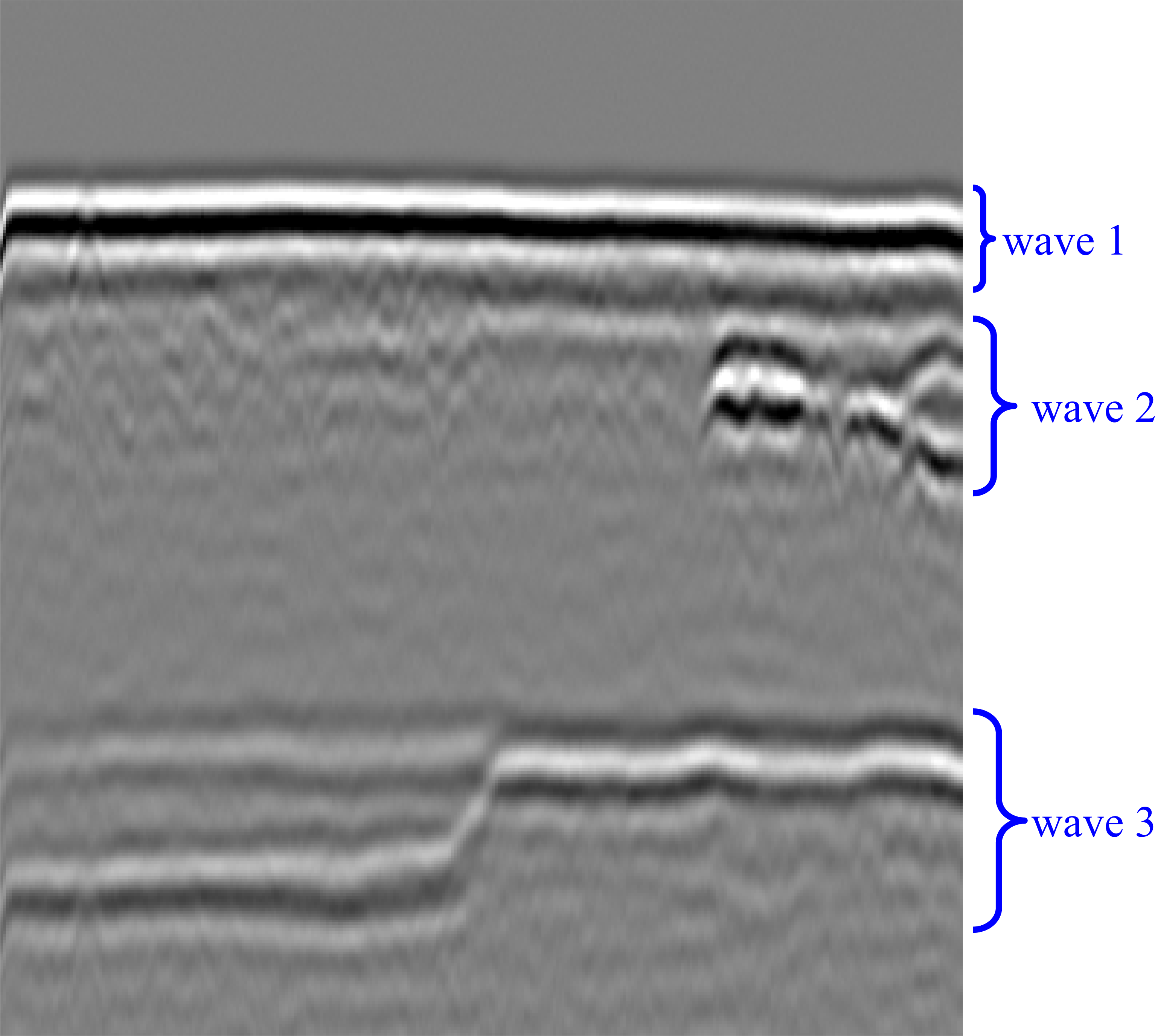}
\label{fig:realData_ndt}}
\subfloat[]{
\includegraphics[height=.4\columnwidth]{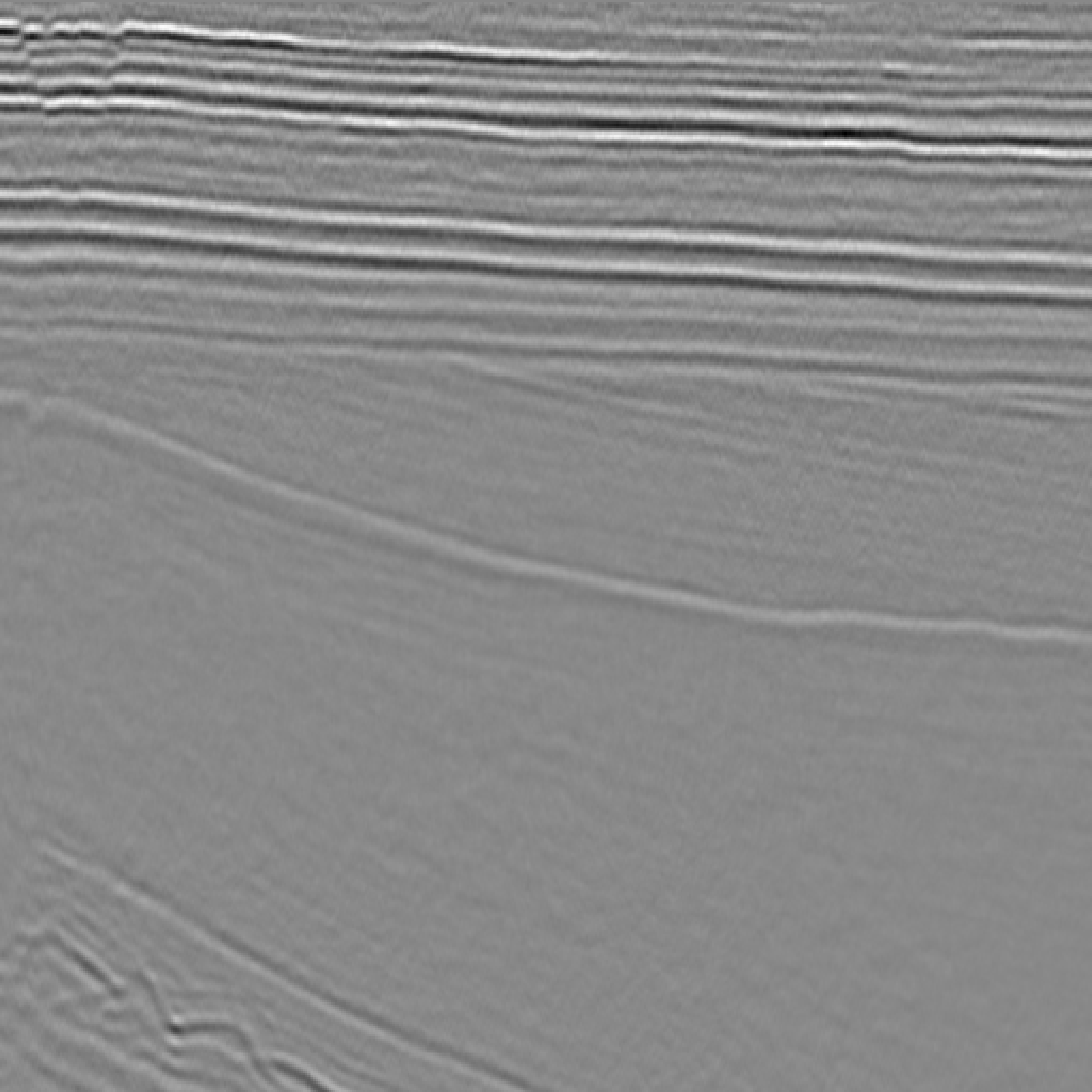}
\label{fig:realData_seis}}
\caption{Data from real world application: (a) ultrasonic non-destructive testing with three waves (objects)\protect\footnotemark[4]\ and (b) seismic exploration.}
\label{fig:realData}
\end{figure}
\footnotetext[4]{Image created using Curly brace annotation by Pål Næverlid Sævik (2022), MATLAB Central File Exchange, (\url{https://www.mathworks.com/matlabcentral/fileexchange/38716-curly-brace-annotation}).}

To obtain reasonable results with ORKA in applications, we first need to find suitable parameters $C$, $K$, and $\mu$. This can be done using the following strategies. The parameter $C$ gives the maximal shift difference between two neighboring columns. It depends on the data resolution and measurements setup. We demonstrate its calculation for the example of seismic data. Let the seismic sensors that recorded the data of two neighboring columns in $\bm{D}$ be placed exactly $1$km apart from each other. A seismic wave (P-wave) moves at a speed of at least $5$km per second. This means, if we record a wave at the first sensor, then we expect the wave to arrive at the second sensor within the next $0.2$ seconds.
(It can arrive earlier when e.g., the source of the wave is directly in between both sensors. But it can not arrive later, if we assume a minimum speed of $5$km per second.) Now let the sensors sampling rate be $5$ samples per second, i.e., one sample every $0.2$ seconds. Then the wave should be recorded at the second sensor with the next sample point, i.e., we choose $C=1$. However, when the sampling rate is increased to one sample every $0.1$ seconds, then we need to set $C=2$. Next, we have to choose the approximation rate $K$. As $K$ balances approximation error and complexity of the algorithm, it can just be set as high as possible until either the runtime of the algorithm is too long or we run out of memory. Last, we need to fix $\mu$. This parameter should be chosen according to the noise on the data. The higher the noise, the more regularization we need to apply. Unfortunately, we can not give a formula or sophisticated strategy to find $\mu$. In our experiments we always tried different parameters $\mu$ until we found the correct range. We suggest to test different parameter choices on small example data. The results only vary slowly with dependence on $\mu$ so a good value can be found easily and finding the precise and optimal $\mu$ is not necessary. Moreover, once a good value is found, the same $\mu$ can be used for any data $\bm{D}$ that was obtained under the same or a similar setup.

In our first experiment, we test ORKA against our previous method Shifted rank-1 approximation to show the advantages of the new model. In \cite{Bossmann20} ultrasonic data was most suitable for the shifted rank-1 model among all tested data types. Hence, we compare both methods on this application to see if we can gain a benefit from ORKA even when the shifted rank-1 approach is already very fitting. We repeated the object separation segmentation experiment (Section 5.4 and Figure 7 in \cite{Bossmann20}) in which we try to decompose the data Figure \ref{fig:realData_ndt} into its single components. At this point, it is important to note that both, the shifted rank-1 approach and ORKA, separate objects by their moving pattern, i.e., by the shift vector $\bm\lambda$. If two objects require the same shift, then both algorithms will group them into one object. In this example, the surface wave and the inner surface reflection will form one object as they run nearly parallel to each other.

In Figure \ref{fig:ultrasound} the first two reconstructed objects using the shifted rank-1 approach and the new ORKA method are shown. We used the parameters $C=1$, $K=5$, and $\mu=500$. Both methods were performed iteratively, i.e., the second object was recovered from the residual that remained after the first object is extracted. Both methods succeed in separating surface signals from the tube defect. However, we can see that the shifted rank-1 approach is not flexible enough to detect the defect in the inner surface while the new ORKA algorithm is. The defect is reconstructed in more detail when using ORKA, but also does slowly fade out over the data instead of being perfectly localized as the shifted rank-1 reconstruction is. This is due to the $2$-norm used in the total change (see Definition \ref{def:object}). The $2$-norm favors smooth transitions instead of hard cut-offs. A better option would be to choose the $1$-norm in Definition \ref{def:object}. However, this would make the minimization problem non-differentiable and much more complex as it already is.

\begin{figure*}
\centering
\subfloat[]{
\includegraphics[height=.22\textwidth]{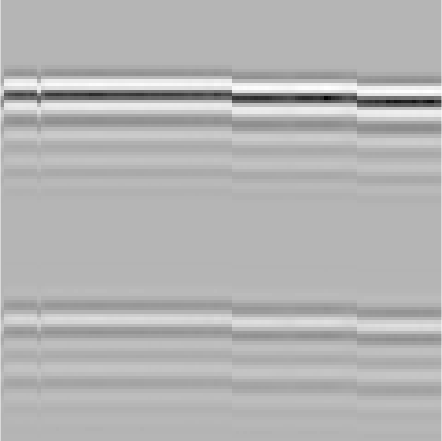}
}\subfloat[]{
\includegraphics[height=.22\textwidth]{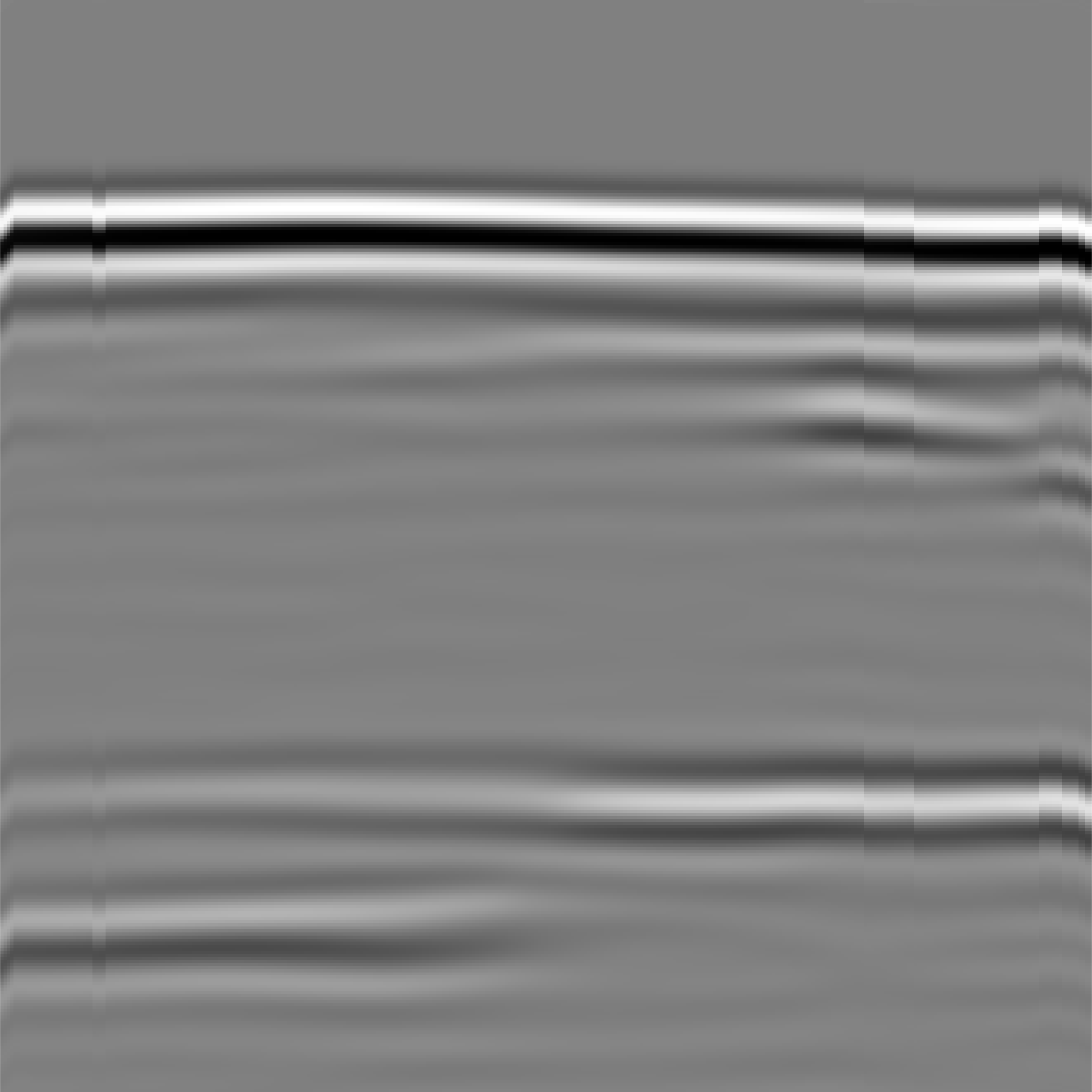}
}\subfloat[]{
\includegraphics[height=.22\textwidth]{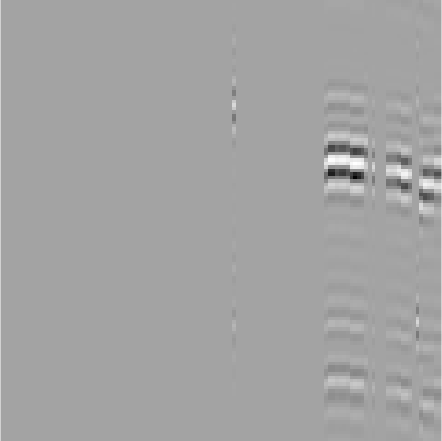}
}\subfloat[]{
\includegraphics[height=.22\textwidth]{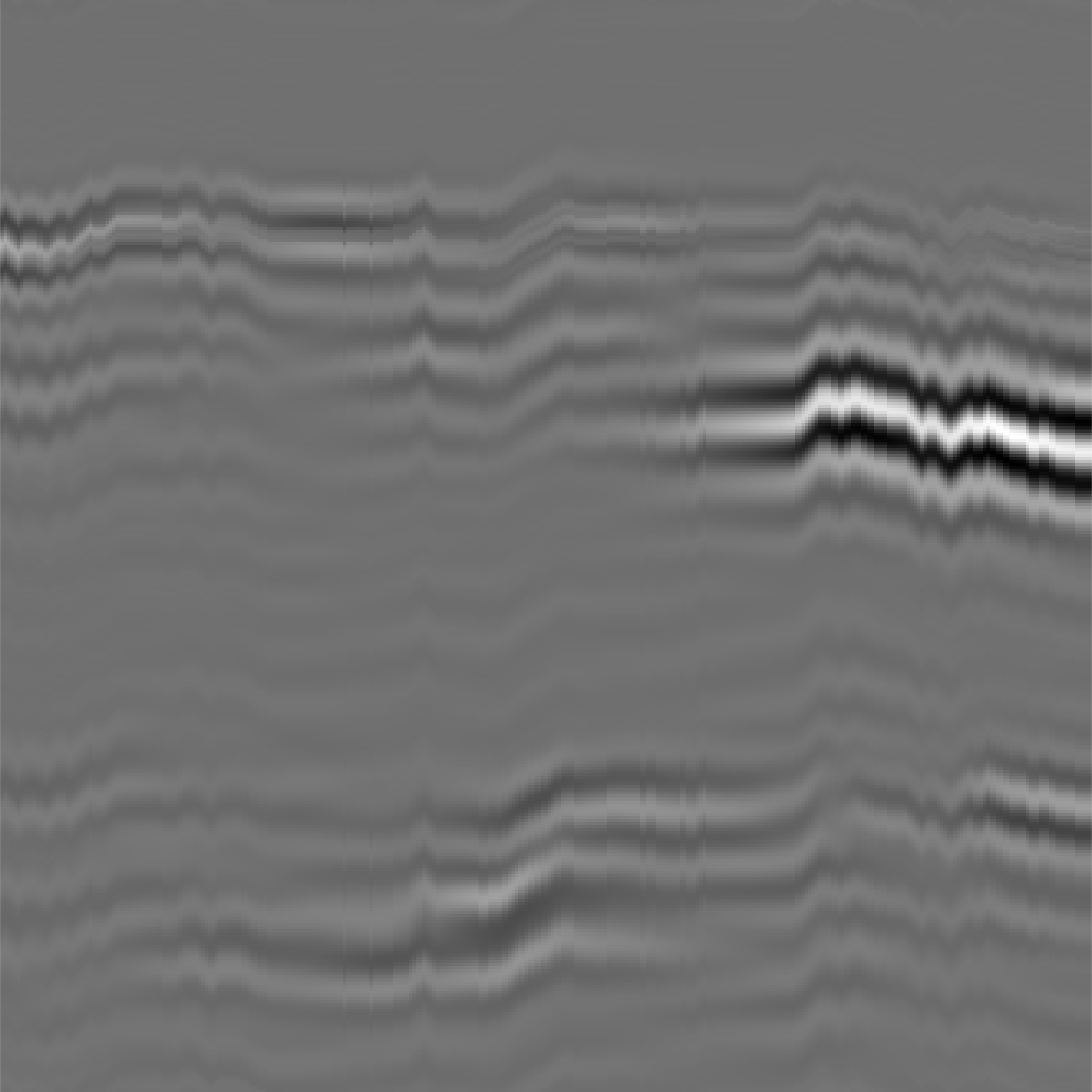}}
\caption{Reconstructed objects using Shifted rank-1 (a,c) and ORKA (b,d). The first object (a,b) is the outer and inner surface wave. The second object (c,d) is a defect in the tube.}
\label{fig:ultrasound}
\end{figure*}

In our next experiment we apply ORKA to highly noised seismic data (Figure \ref{fig:realData_seis}). We test if the algorithm is still able to track the seismic waves in the data even when its hidden under strong noise. We then take the sum of all found waves as denoised version of the signal. We compare the results to Shearlet denoising \cite{Kutyniok16} (using ShearLab3Dv1.1 toolbox for MatLab \url{www.shearlab.org}). Shearlets are especially useful for sparse approximation and denoising of many kinds of data and they have been used for seismic denoising before \cite{Kong15}. In Figure \ref{fig:PSNR_graph} we plotted the obtained PSNR value after denoising with ORKA and Shearlets against the noised PSNR value. Both algorithms were iterated over a reasonable number of parameter choices and the best result was chosen in the end. We can see that Shearlets are better in the low to mid noise level case (PSNR of 12 to 20). For a high noise level (PSNR of 6 to 12) both methods are about equal. However, for very high noise (PSNR below 6), ORKA yields better denoising results. ORKA tracks a seismic wave through all measurements and thus can use global information to denoise the data while Shearlets have a certain locality and get overwhelmed by the noise. We demonstrate this in Figure \ref{fig:PSNR_path} where the tracked movement of a wave is plotted on top of the original data. We can see that the path clearly follows the original wave, although it is slightly corrupted by the noise. The noisy data this path was obtained from is shown in Figure \ref{fig:PSNR_noise}, where one can barely see the seismic data hidden under the noise. (Also compare this to Figure 8 in \cite{Bossmann20} where the same experiment is performed but without Lipschitz-continuity restriction for the path.)

\begin{figure*}
\centering
\subfloat[]{
\includegraphics[height=.22\textwidth]{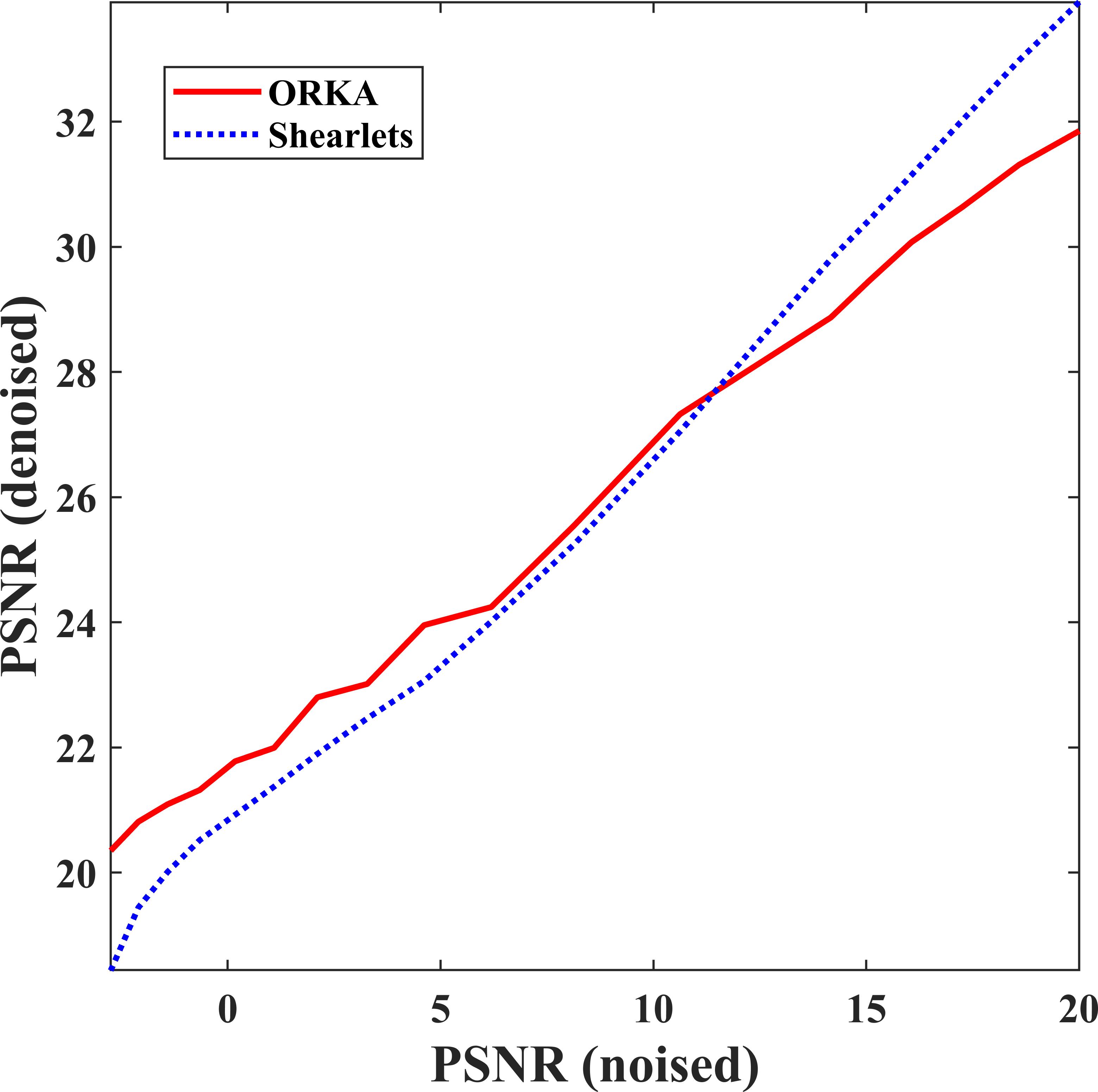}
\label{fig:PSNR_graph}}
\subfloat[]{
\includegraphics[height=.22\textwidth]{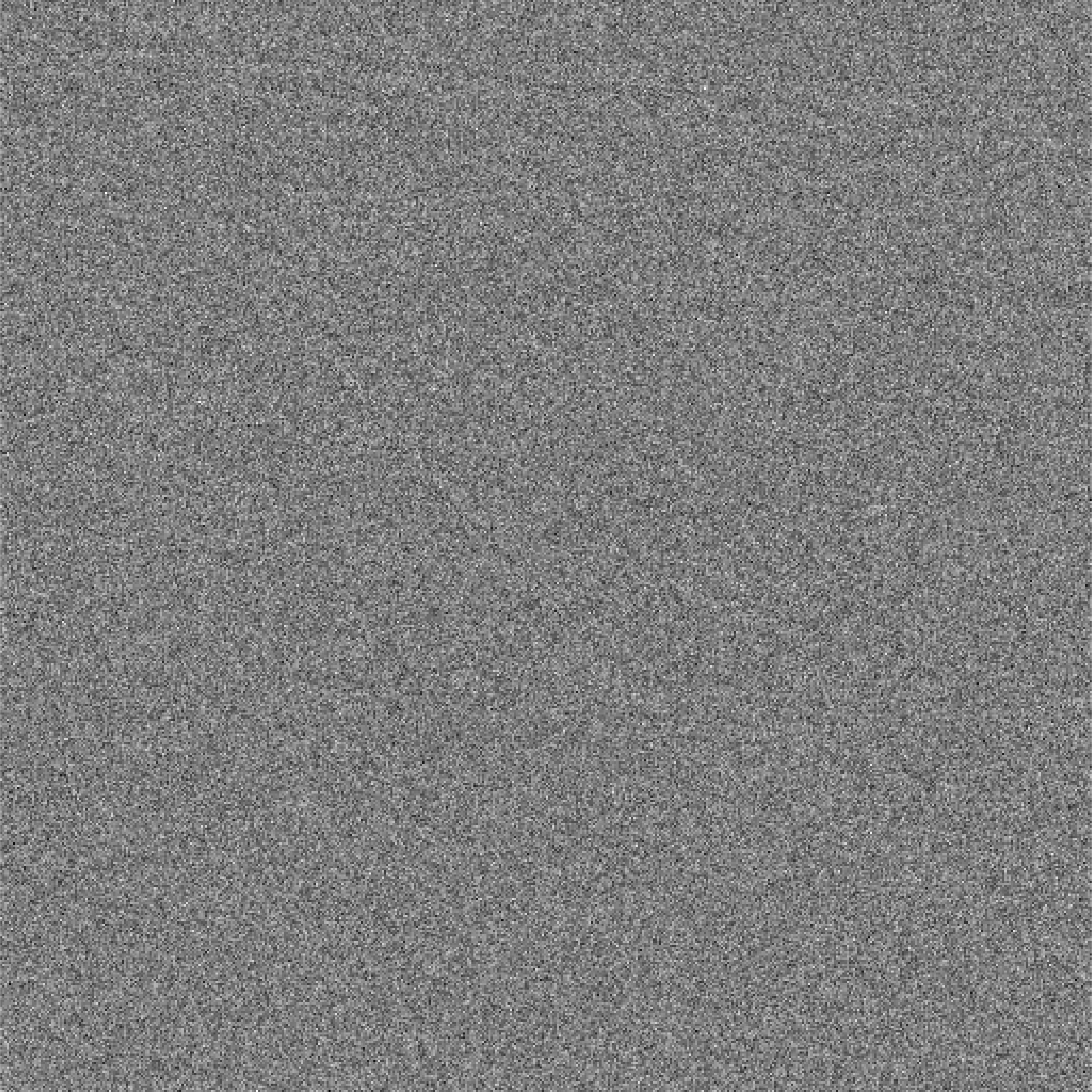}
\label{fig:PSNR_noise}}
\subfloat[]{
\includegraphics[height=.22\textwidth]{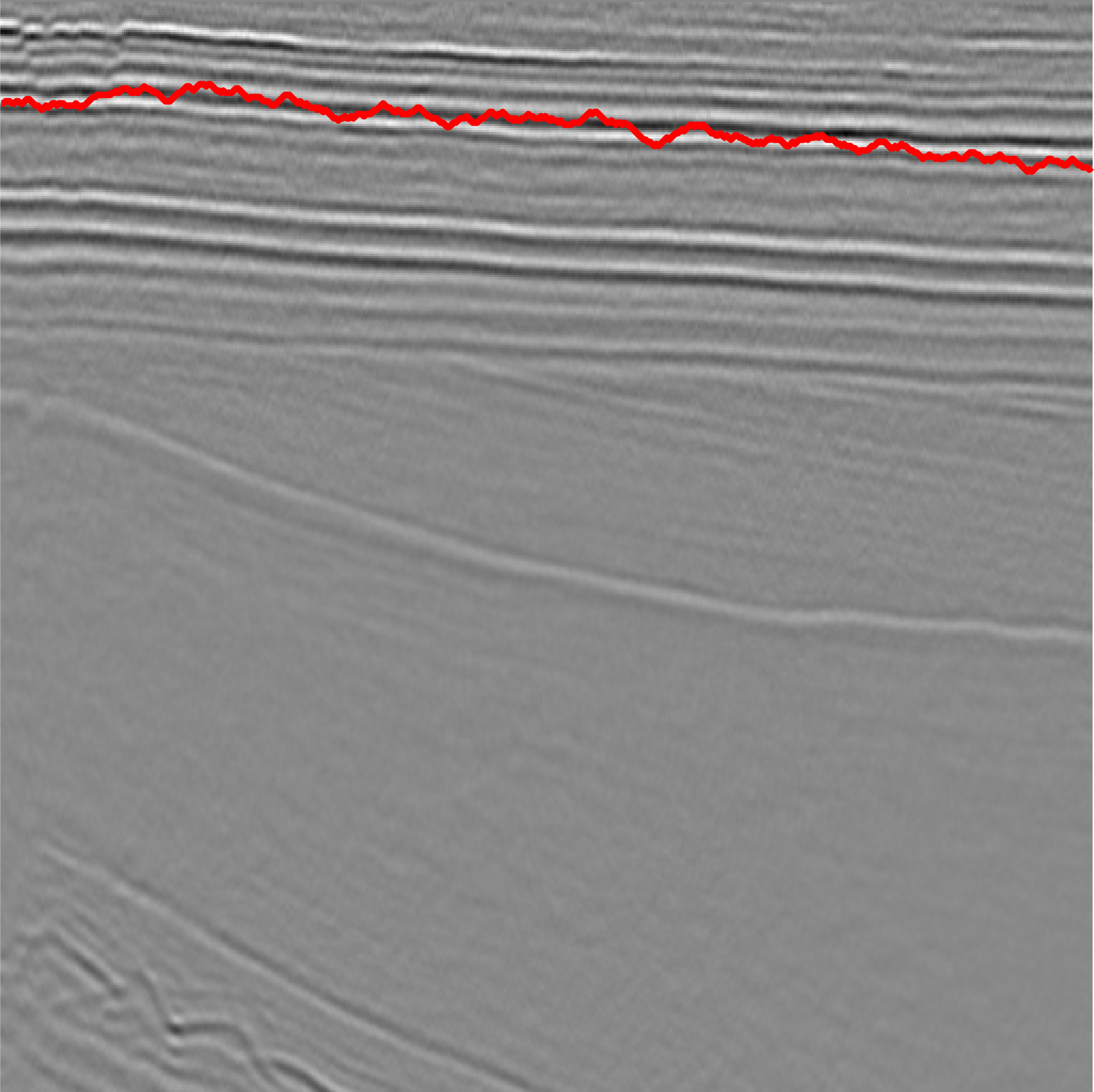}
\label{fig:PSNR_path}}
\caption{Seismic denoising using ORKA: (a) PSNR value of denoised image vs PSNR value of noisy version for ORKA and Shearlets; (b) highly noised seismic data with PSNR close to $0$; (c) tracked seismic wave (red line) on top of original data.}
\end{figure*}

\subsection{Video data}

In this section we apply ORKA to video data. The used videos are presented in Figure \ref{fig:videos} by their first and last frame. We choose two videos from the "UCF Sports Action Data Set" from the UCF Center of Research in Computer Vision \cite{Rodriguez08,Soomro14}. The first video shows a simple scene where a person is running from left to right with a building in the background. The second video is a much more complicated scene from a football match. Here, several players, the ball, and the background are all moving while the camera is rotating from left to right. The background, persons, and ball form the objects of each video that we want to extract using ORKA. The data $\bm{D}$ is the video, which is three dimensional now.
	
\begin{figure*}
\centering
\subfloat[]{\begin{minipage}{.4\textwidth}
\includegraphics[width=\textwidth]{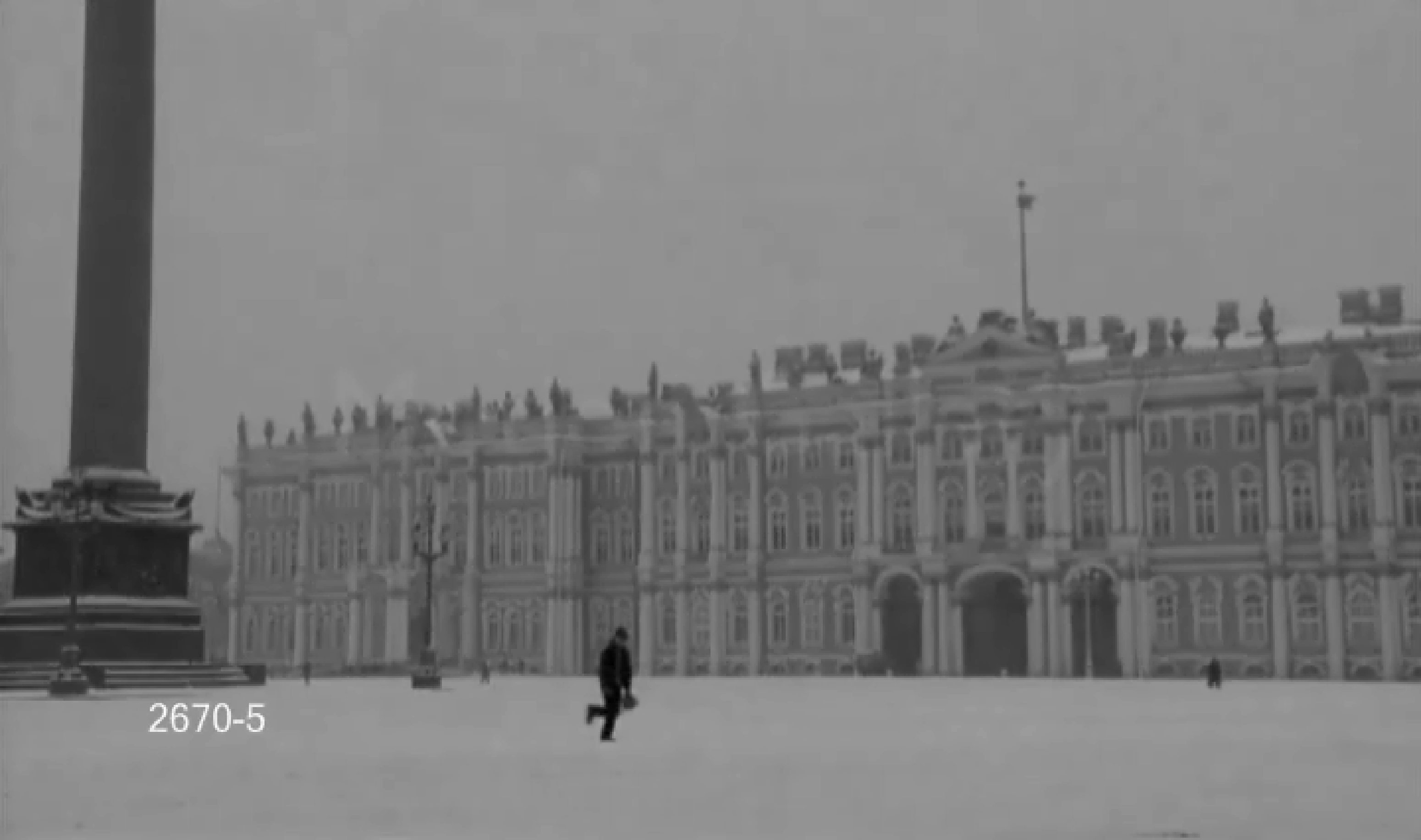}\\[0.1cm]
\includegraphics[width=\textwidth]{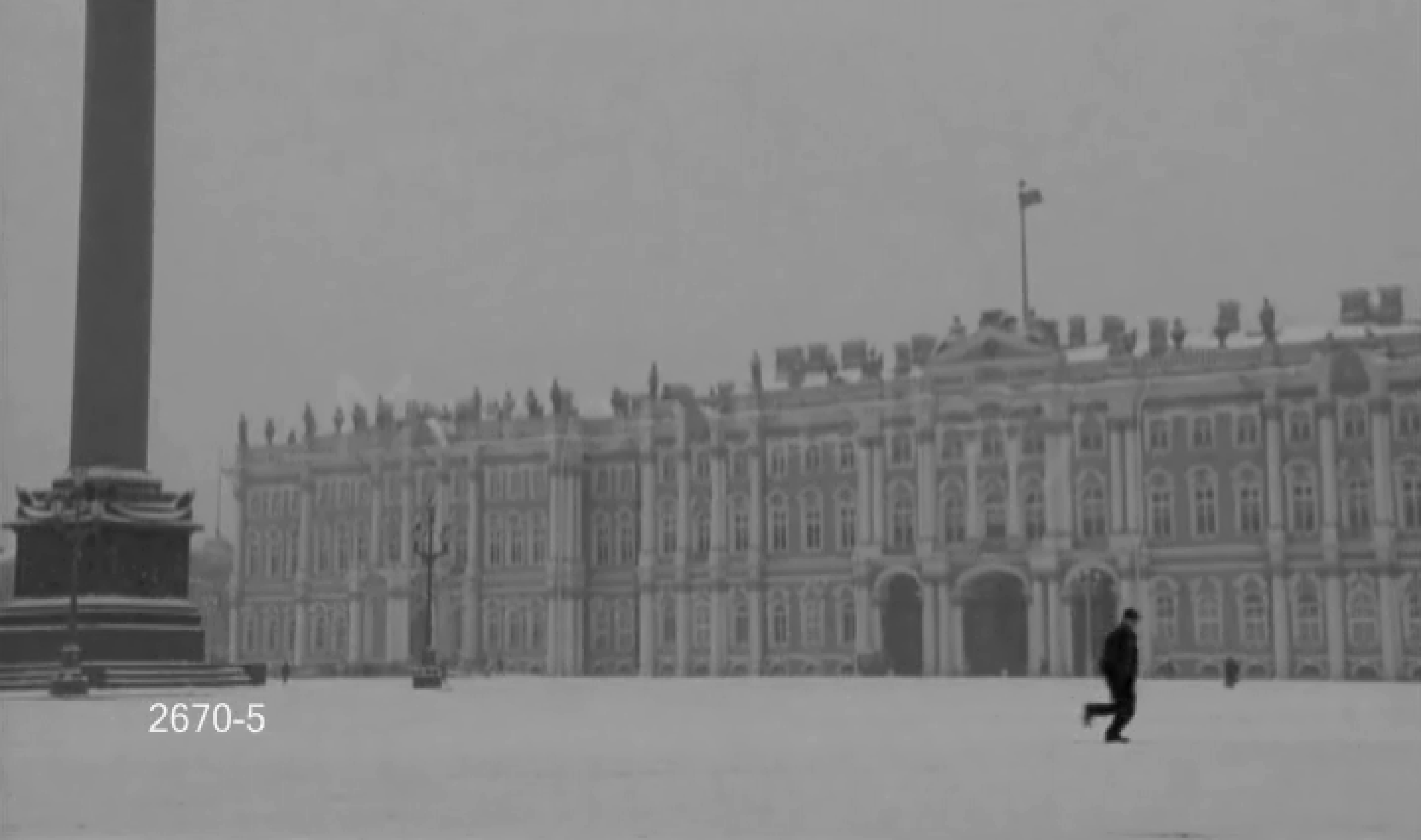}
\end{minipage}\label{fig:videos_run}}\hspace*{0.1cm}
\subfloat[]{\begin{minipage}{.4\textwidth}
\includegraphics[width=\textwidth]{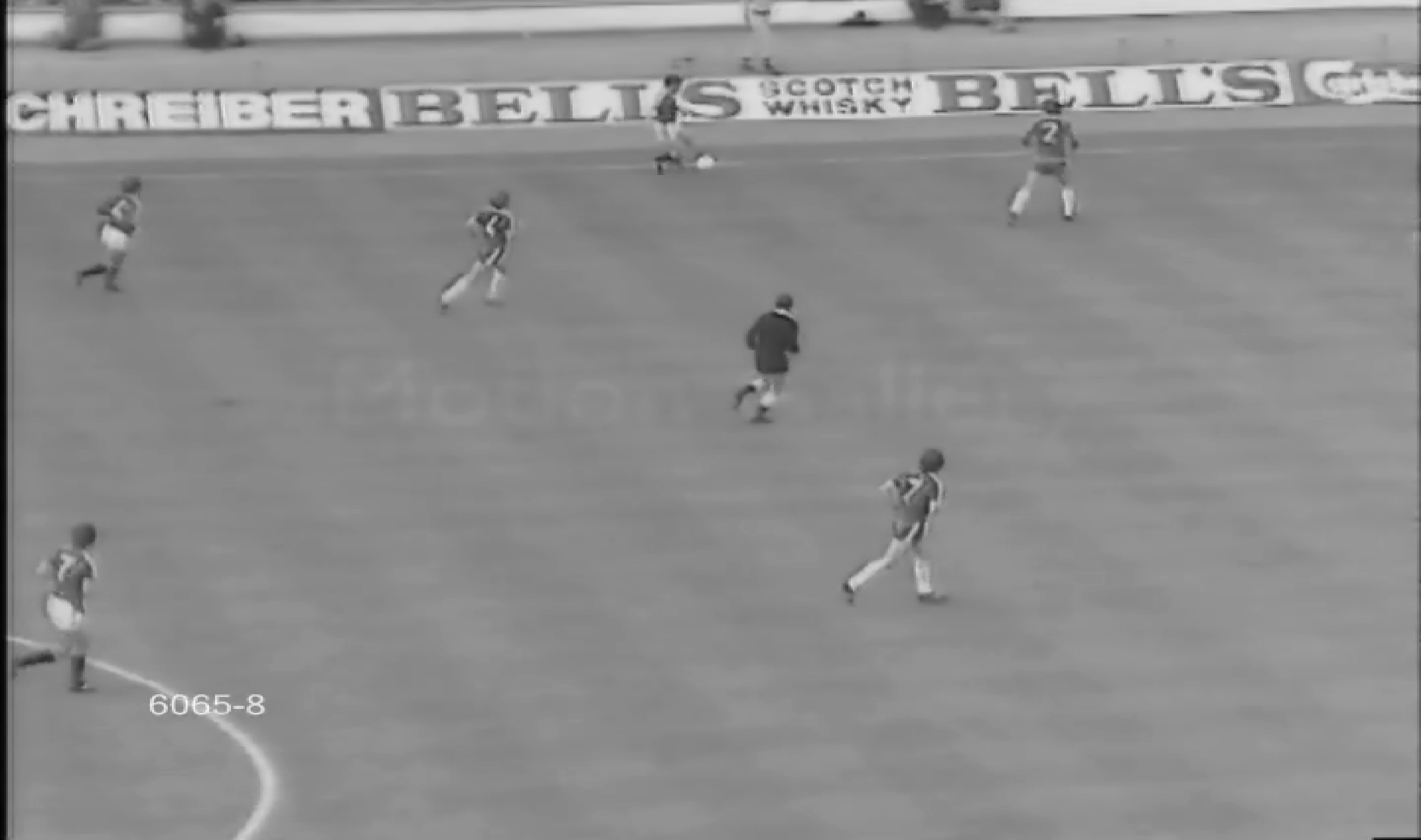}\\[0.1cm]
\includegraphics[width=\textwidth]{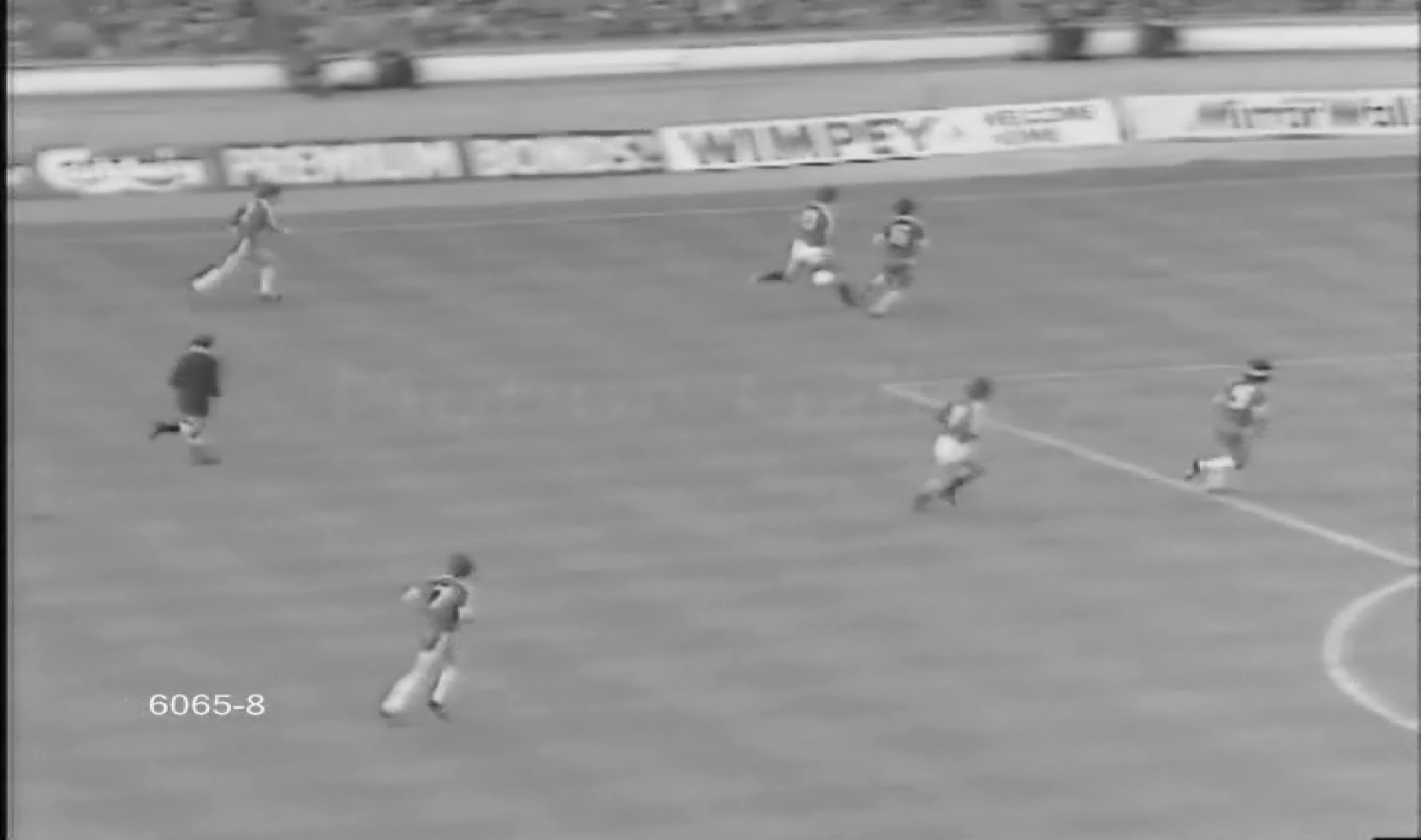}
\end{minipage}\label{fig:videos_soccer}}
\caption{Videos from "UCF Sports Action Data Set" used in this section: (a) running person with constant background and (b) football match scene. The images show the first and the last frame of the video.}
\label{fig:videos}
\end{figure*}
	
For video processing, we consider each frame as a measurement. Different from the 2D data case, an object in a video can move in two dimensions. Hence, we need to update our method for this case. This is actually quite simple by considering two shift vectors $\bm{\lambda}^x$ and $\bm{\lambda}^y$ that give the shift in $x$- or $y$-direction separately Each shift vector must be Lipschitz continuous with constant $C$. For each frame, we now have $(2C+1)^2$ possible movements. This means, there are $(2C+1)^{2(K-1)}$ vertices per graph partition and the complexity of ORKA goes up to $(2C+1)^{2K}$. Hence, we need to choose a smaller $K$ to still have a decent performance.

First, we use ORKA on the running person video (Figure \ref{fig:videos_soccer}). This video consists of two objects: the background and the person. Hence, we use two iterations to get both objects. As the runner moves quite fast, we need to choose $C=5$ and thus $K=4$ resulting in more than $200$ million nodes per graph partition. We use $\mu=\infty$ for the first iteration to force a constant object (static background). To reconstruct the running person we need to choose a small $\mu=1$ as the person changes its outline while moving. The reconstructions are shown in Figure \ref{fig:runner}. ORKA perfectly separated background and person. Figure \ref{fig:runner_f1} and \ref{fig:runner_f2} show the reconstructed running person in two different frames. Note that the outline in both frames is different (e.g., the legs position). Capturing this change in form is only possible due to the model used in this work and was not possible with the shifted rank-1 approach.

\begin{figure*}
\centering
\subfloat[]{
\includegraphics[width=.4\textwidth]{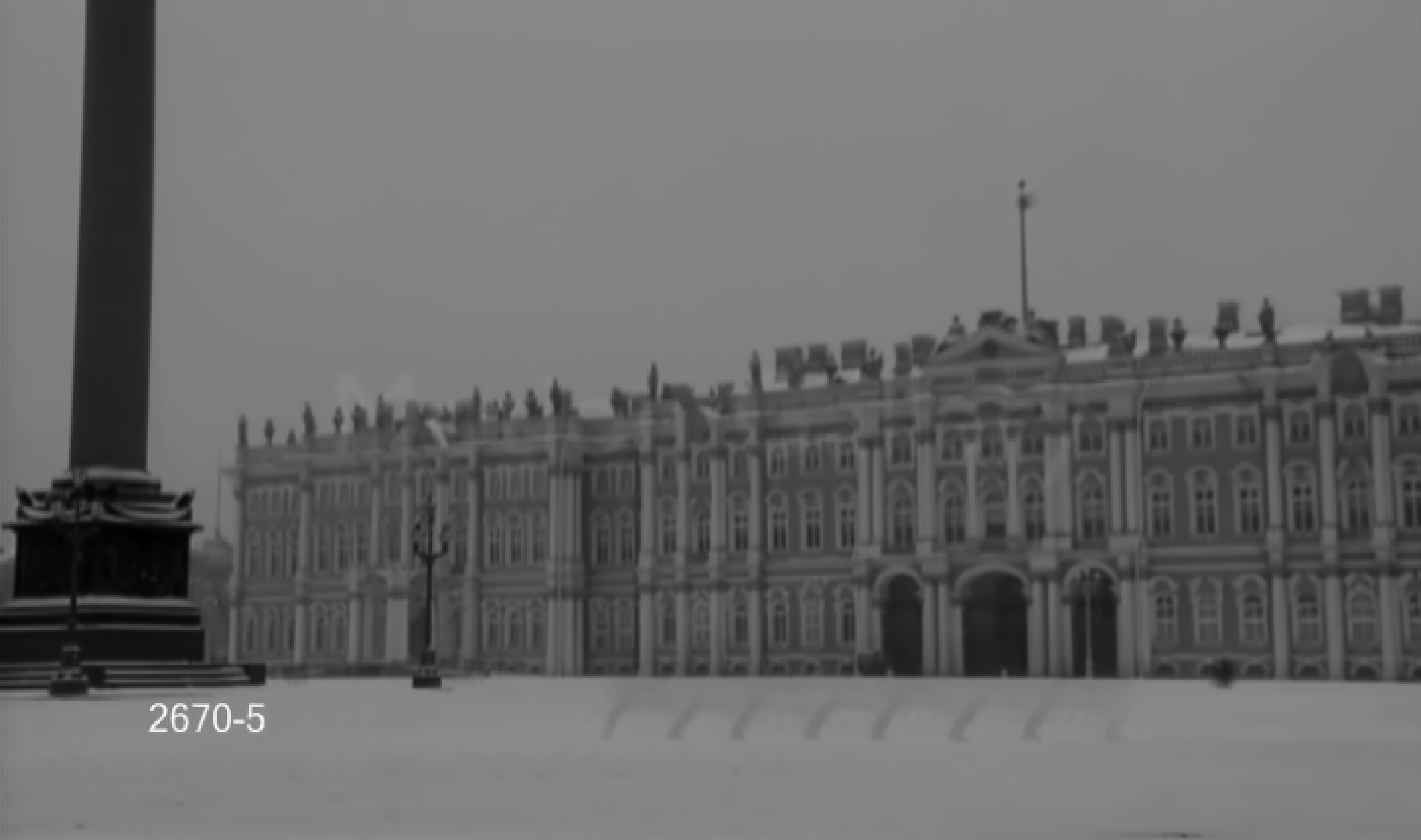}}\\[0.1cm]
\subfloat[]{
\includegraphics[width=.4\textwidth]{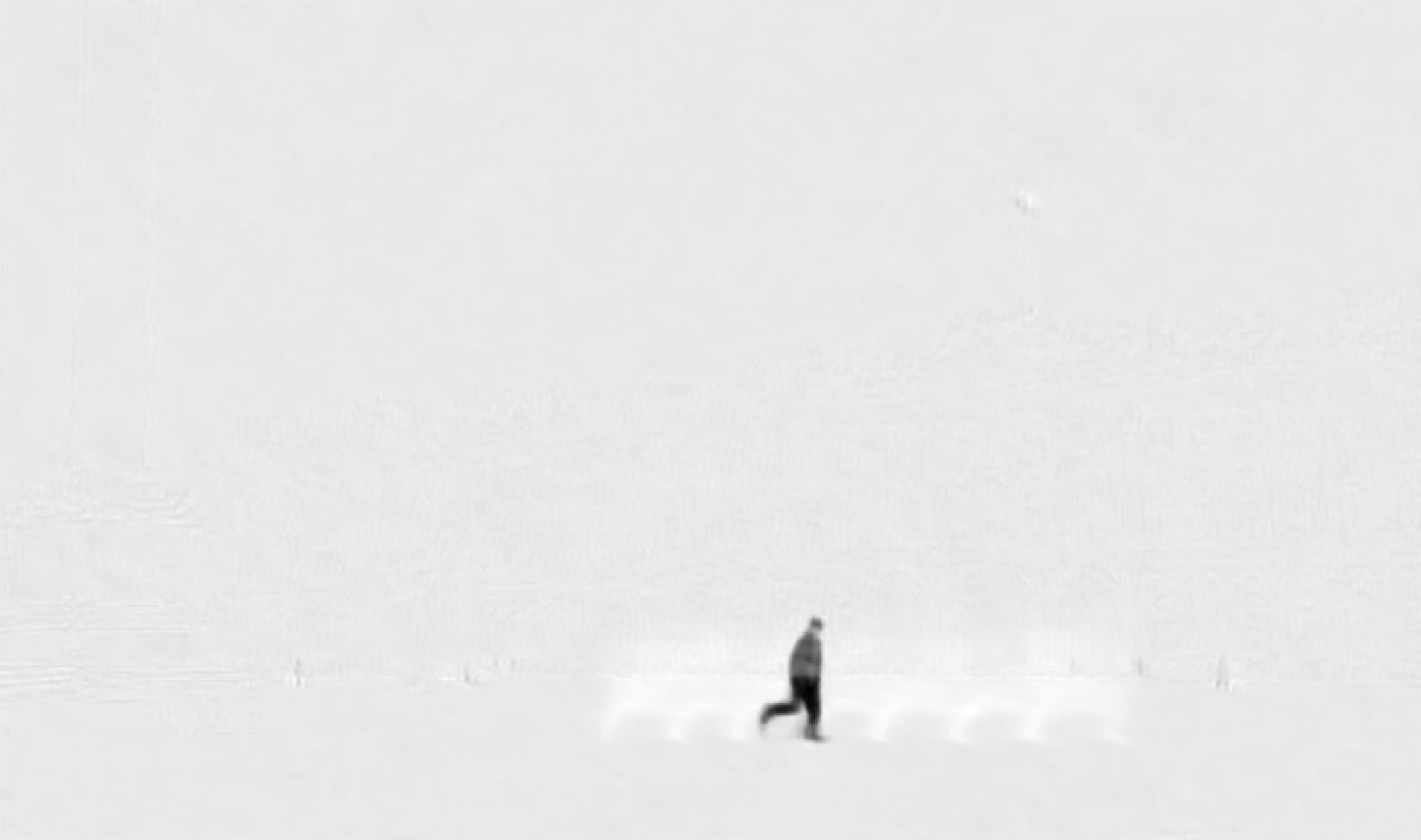}
\label{fig:runner_f1}}
\subfloat[]{
\includegraphics[width=.4\textwidth]{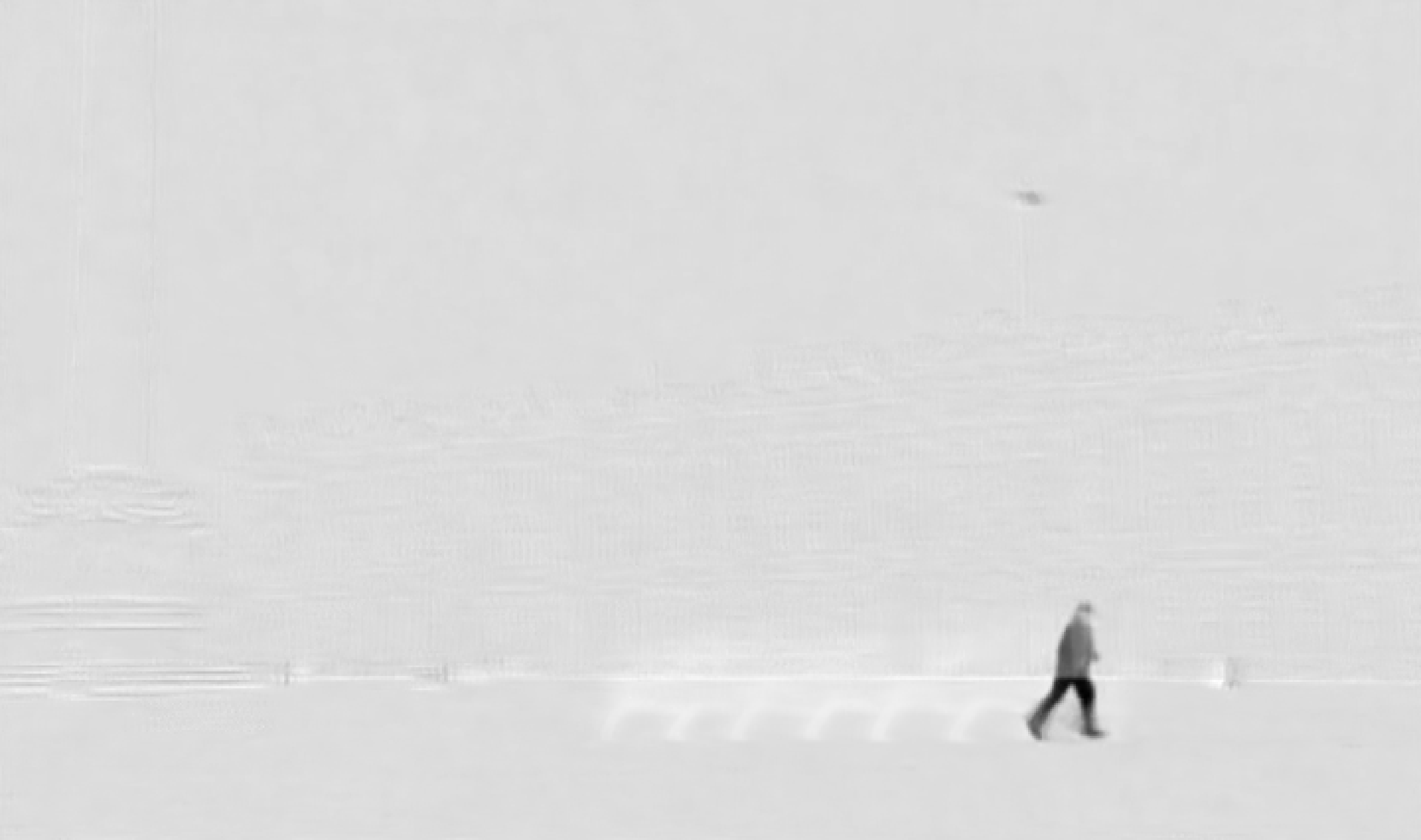}
\label{fig:runner_f2}}
\caption{(a) Reconstructed background and (b,c) person of the from video Figure \ref{fig:videos_run} using ORKA. The changing outline of the runner can be captured as e.g., the legs in both frames (b,c) are in a different position.}
\label{fig:runner}
\end{figure*}

The football video shows a much more complicated scenery. First, as players and camera are moving very fast, we need to choose a high Lipschitz constant $C=12$. We choose $K=3$ leading to $25^6\approx244$ million nodes per partition. Increasing $K$ by one would lead to $625$ times more nodes and is way to complex for the system in use. Still, with this setup the calculation took less than two minutes on an Intel i7-10750H 2.6Ghz with only 16GB memory. The next problem is setting the parameter $\mu$. As the players and background are constantly changing due to the camera and player movement, we would like to choose a small parameter. On the other hand, a small $\mu$ could lead to different objects fading into each others or even not being separated at all. After some testing we stick to $\mu=500$ as the best choice. The first three objects obtained with ORKA in this parameter setting are shown in Figure \ref{fig:football}. The first object is again the background. This time, it is the general background pattern that was found without showing any details. We also find the time stamp of the camera here and some blurring effects from the moving players. The second object contains the perimeter advertising as well as the field pattern. These are the static parts of the scenery that only move in the video whenever the camera is moving. The third reconstructed object contains the players and referee. As we were watching the video closely, it became clear that all players and the referee are running at about the same speed in about the direction. ORKA detects the objects by their movement pattern. The algorithm does not force any locality constraints like compact support on the object. Thus, it clusters everything that moves at about the same speed (and direction) into one object, like perimeter advertising and field pattern, or all players and referee. We also note that some players can be seen more clearly than others. Those players are more inline with the calculated object movement (i.e., the average players speed and direction). Including more constraints into the model, such as forcing compact support on $\bm{U}$, could help to avoid this clustering.

\begin{figure*}
\centering
\subfloat[]{
\includegraphics[width=.4\textwidth]{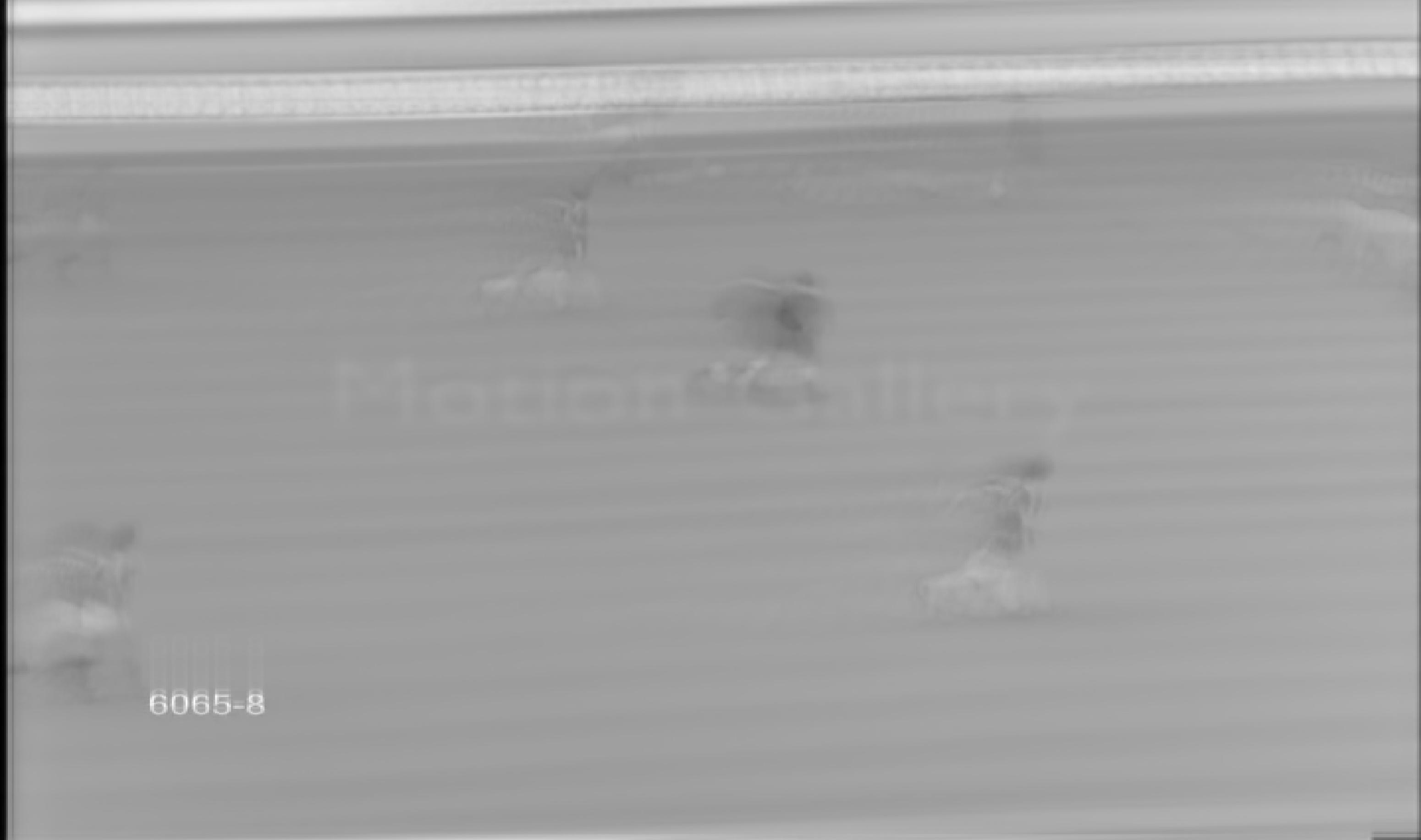}
}\\
\subfloat[]{
\includegraphics[width=.4\textwidth]{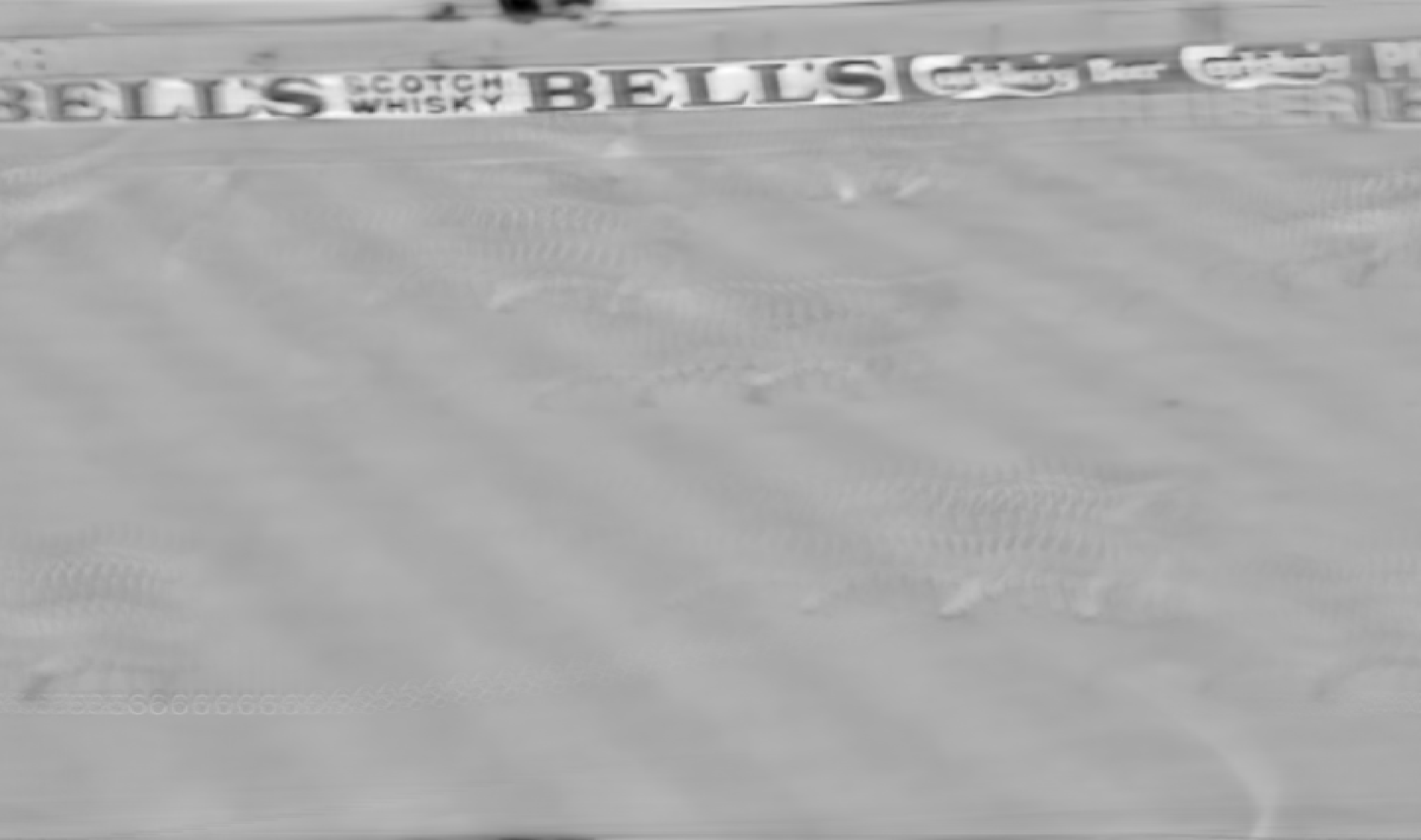}
}\subfloat[]{
\includegraphics[width=.4\textwidth]{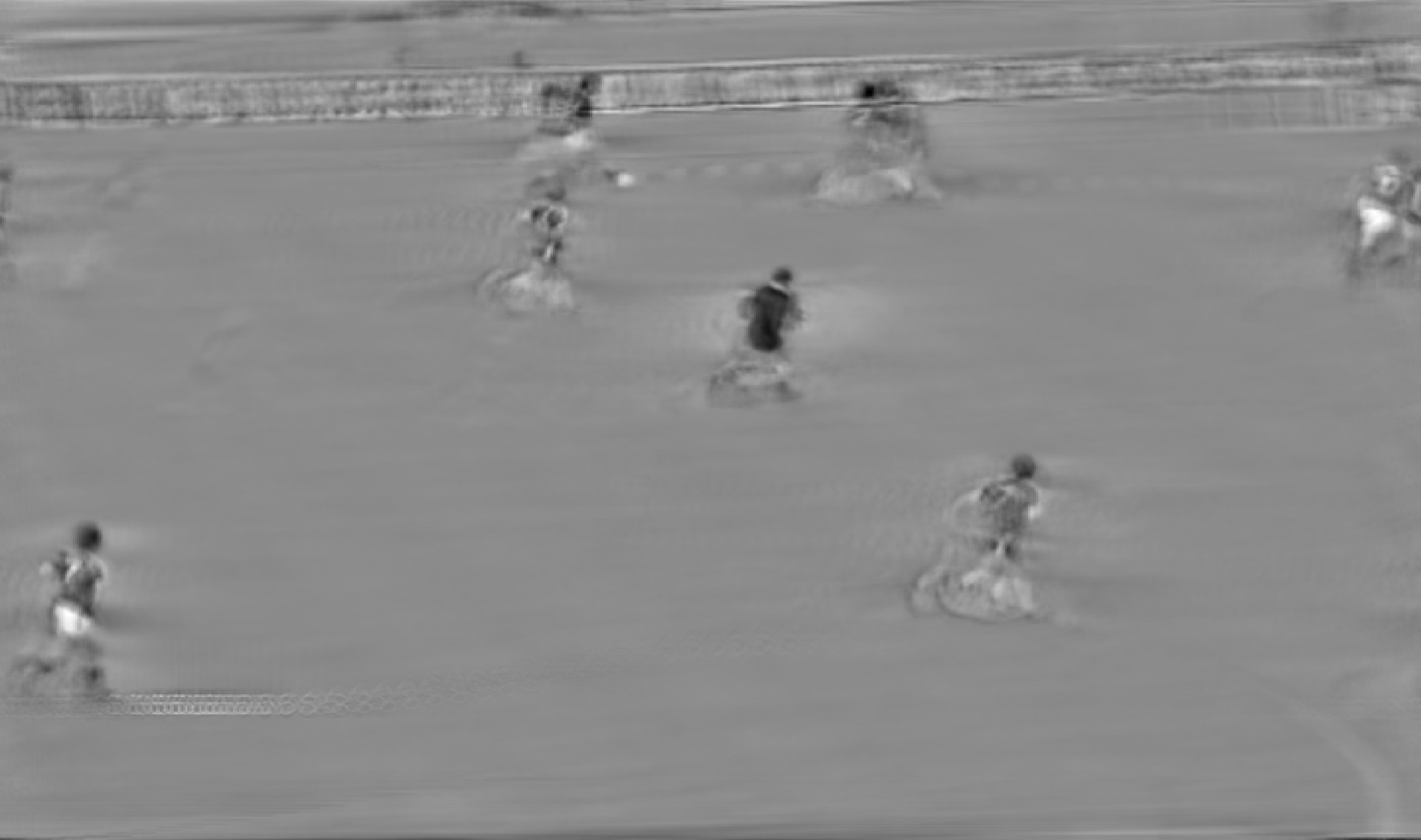}}
\caption{Reconstructed objects after three iterations of ORKA on the football video: (a) general background, (b) perimeter advertising, and (c) players/referee.}
\label{fig:football}
\end{figure*}

\section{Conclusion}

We have presented a new algorithm for object reconstruction in multiple measurements. Therefore, we modeled an object as any kind of structure that can move through the measurements, i.e., it can change position and form from one measurement to the other. Thereby, the model assures that the path formed by the object positions is Lipschitz-smooth and the total change in form over all measurements is bounded. These assumptions are reasonable for many applications. We analyzed the according combinatorial optimization problem and gave an approximate solution based on a $K$-approximation graph. The presented algorithm then solved the longest-path problem on the constructed graph. We were also able to prove approximation error bounds.

We numerically confirmed the theoretical complexity and approximation behavior of our technique. Moreover, we demonstrated the benefits of the new method on practical applications: seismic exploration and non-destructive testing. Last, we also showed that the algorithm can also be applied to multi-dimensional data such as video processing.

While the model is already quite sophisticated, we still are eager to improve it in the future. Right now, the recovered objects are not limited in size. This often leads to a blurry object boundary instead of sharp contours, or several objects with the same movement pattern being grouped together. Hence, forcing compact support on the object, or switching from $2$-norm to $1$-norm when calculating the total change are interesting future research topics. Also, taking the difference between more than one neighboring column in the total change into account could lead to more sophisticated structures. Restrictions others than Lipschitz continuity on the path might also be reasonable depending on the application. Last, the exponential runtime of the algorithm for large $C$ is a drawback especially for multi-dimensional data and needs to be addressed further.

However, we think that the answers to the above considerations are highly dependent on the concrete application task. This work is therefore also a proof of concept that these kind of combinatorial quadratic optimization problems can be handled efficiently and the resulting techniques are viable. It can be used as blueprint to transfer the ideas to specific tasks and refine the model by including any given a-priori information about the structure or movement of the objects.

\bibliographystyle{unsrt}
\bibliography{references}

\begin{thebibliography}{10}

\bibitem{Mallat99}
S.~Mallat.
\newblock {\em A wavelet tour of signal processing}.
\newblock Elsevier, 1999.

\bibitem{Zhou14}
H.~Zhao X.~Zhou, C.~Yang and W.~Yu.
\newblock Low-rank modeling and its applications in image analysis.
\newblock {\em ACM Computing Surveys}, 47(2):1--33, 2014.

\bibitem{Foucart13}
S.~Foucart and H.~Rauhut.
\newblock {\em A Mathematical Introduction to Compressive Sensing}.
\newblock Birkhauser Basel, 2013.

\bibitem{Chen06}
J.~Chen and X.~Huo.
\newblock Theoretical results on sparse representations of multiple-measurement
  vectors.
\newblock {\em IEEE Trans. on Signal Processing}, 54(12):4634--4643, 2006.

\bibitem{Do09}
T.~T. Do, Y.~Chen, D.~T. Nguyen, N.~Nguyen, L.~Gan, and T.~D. Tran.
\newblock Distributed compressed video sensing.
\newblock {\em 16th IEEE Int. Conf. on Image Processing}, 16:1393--1396, 2009.

\bibitem{Ma10}
J.~Ma and G.~Plonka.
\newblock A review of curvelets and recent applications.
\newblock {\em IEEE Signal Processing Magazine}, 27(2):118--133, 2010.

\bibitem{Kutyniok12}
G.~Kutyniok and D.~Labate.
\newblock {\em Shearlets: Multiscale analysis for multivariate data}.
\newblock Springer Science and Business Media, 2012.

\bibitem{Wu14}
Y.~Wu, Y.-J. Zhu, Q.-Y. Tang, C.~Zou, W.~Liu, R.-B. Dai, X.~Liu, E.~X. Wu,
  L.~Ying, and D.~Liang.
\newblock Accelerated mr diffusion tensor imaging using distributed compressed
  sensing.
\newblock {\em Magnetic Resonance in Medicine}, 71(2):764--–772, 2014.

\bibitem{Eldar10}
P.~Kuppinger Y.~C.~Eldar and H.~Bolcskei.
\newblock Block-sparse signals: Uncertainty relations and efficient recovery.
\newblock {\em IEEE Transactions on Signal Processing}, 58(6):3042--3054, 2010.

\bibitem{Huang11}
T.~Zhang J.~Huang and D.~Metaxas.
\newblock Learning with structured sparsity.
\newblock {\em Journal of Machine Learning Research}, 12(11), 2011.

\bibitem{Wen16}
W.~Wen, C.~Wu, Y.~Wang, Y.~Chen, and H.~Li.
\newblock Learning structured sparsity in deep neural networks.
\newblock {\em Advances in neural information processing systems}, 29, 2016.

\bibitem{Jia12}
T.~H.~Chan K.~Jia and Y.~Ma.
\newblock Robust and practical face recognition via structured sparsity.
\newblock {\em European conference on computer vision}, pages 331--344, 2012.

\bibitem{Wu20}
Y.~Chen H.~Wu, S.~Li and Z.~Peng.
\newblock Seismic impedance inversion using second-order overlapping group
  sparsity with a-admm.
\newblock {\em Journal of Geophysics and Engineering}, 17(1):97--116, 2020.

\bibitem{Yu16}
J.~Ma S.~Yu and S.~Osher.
\newblock Monte carlo data-driven tight frame for seismic data recovery.
\newblock {\em Geophysics}, 81(4):V327--V340, 2016.

\bibitem{Tosic11}
I.~Tosic and P.~Frossard.
\newblock Dictionary learning.
\newblock {\em IEEE Signal Processing Magazine}, 28:27--38, 2011.

\bibitem{Plonka10}
S.~Tenorth G.~Plonka and D.~Rosca.
\newblock A new hybrid method for image approximation using the easy path
  wavelet transform.
\newblock {\em IEEE Transactions on Image Processing}, 20(2):372--381, 2010.

\bibitem{Bossmann21}
J.~Maly F.~Bossmann, S. Krause-Solberg and N.~Sissouno.
\newblock Structural sparsity in multiple measurements.
\newblock {\em IEEE Trans. on Signal Processing}, 70:280--291, 2021.

\bibitem{Bossmann20}
F.~Bossmann and J.~Ma.
\newblock Enhanced image approximation using shifted rank-1 reconstruction.
\newblock {\em Inverse Problems and Imaging}, 14(2):267--290, 2020.

\bibitem{Rusu13}
B.~Dumitrescu C.~Rusu and S.~A. Tsaftaris.
\newblock Explicit shift-invariant dictionary learning.
\newblock {\em IEEE Signal Processing Letters}, 21(2):6--9, 2013.

\bibitem{Sundman13}
S.~Chatterjee D.~Sundman and M.~Skoglund.
\newblock Methods for distributed compressed sensing.
\newblock {\em Journal of Sensor and Actuator Networks}, 3(1):1--25, 2013.

\bibitem{Basarab13}
A.~Basarab, H.~Liebgott, O.~Bernard, D.~Friboulet, and D.~Kouame.
\newblock Medical ultrasound image reconstruction using distributed compressive
  sampling.
\newblock {\em 10th IEEE Int. Symp. on Biomedical Imaging}, pages 628--631,
  2013.

\bibitem{Zhou22}
C.~Yi P.~Zhou, L.~He and Q.~Zhou.
\newblock Impulses recovery technique based on high oscillation region
  detection and shifted rank-1 reconstruction—its application to bearing
  fault detection.
\newblock {\em IEEE Sensors Journal}, 22(8):8084--8093, 2022.

\bibitem{Vaswani10}
N.~Vaswani and W.~Lu.
\newblock Modified-cs: Modifying compressive sensing for problems with
  partially known support.
\newblock {\em IEEE Trans. Signal Processing}, 58(9):4595–--4607, 2010.

\bibitem{Heins15}
M.~Moeller P.~Heins and M.~Burger.
\newblock Locally sparse reconstruction using the {$l^{1,\infty}$}-norm.
\newblock {\em Inverse Problems and Imaging}, 9(4):1093–--1137, 2015.

\bibitem{Lian19}
A.~Liu L.~Lian and V.~Lau.
\newblock Exploiting dynamic sparsity for downlink fdd-massive mimo channel
  tracking.
\newblock {\em IEEE Trans. Signal Processing}, 67(8):2007–--2021, 2019.

\bibitem{Wittenburg98}
J.~Wittenburg.
\newblock Inverses of tridiagonal toeplitz and periodic matrices with
  applications to mechanics.
\newblock {\em Journal of Applied Mathematics and Mechanics}, 62(4):575--587,
  1998.

\bibitem{Yueh06}
W-C. Yueh.
\newblock Explicit inverses of several tridiagonal matrices.
\newblock {\em Applied Mathematic E-Notes}, 6:74--83, 2006.

\bibitem{Losonczi92}
L.~Losonczi.
\newblock Eigenvalues and eigenvectors of some tridiagonal matrices.
\newblock {\em Acta Mathematica Hungarica}, 60(3-4):309--322, 1992.

\bibitem{Kutyniok16}
W.-Q.~Lim G.~Kutyniok and R.~Reisenhofer.
\newblock Shearlab 3d: Faithful digital shearlet transforms based on compactly
  supported shearlets.
\newblock {\em ACM Trans. Math. Software}, 42(5), 2016.

\bibitem{Kong15}
D.~Kong and Z.~Peng.
\newblock Seismic random noise attenuation using shearlet and total generalized
  variation.
\newblock {\em Journal of Geophysics and Engineering}, 12(6):1024--1035, 2015.

\bibitem{Rodriguez08}
J.~Ahmed M.~D.~Rodriguez and M.~Shah.
\newblock Action mach: A spatio-temporal maximum average correlation height
  filter for action recognition.
\newblock {\em IEEE Conf. Computer Vision and Pattern Recognition}, pages 1--8,
  2008.

\bibitem{Soomro14}
K.~Soomro and A.~R. Zamir.
\newblock Action recognition in realistic sports videos.
\newblock {\em Springer Computer vision in sports}, pages 181--208, 2014.

\end{thebibliography}

\end{document}